%% file: Ar_Mac_Quasimodes.tex
\documentclass[11pt, a4paper, leqno]{amsart}

\usepackage{esint} 
\usepackage{graphicx}
\usepackage{amsthm}
\usepackage{amsmath}
\usepackage{amsfonts}
\usepackage{amssymb}
\usepackage{amscd}
\usepackage{mathrsfs}
\usepackage{enumerate}
\usepackage{verbatim}

\DeclareFontFamily{U}{mathx}{\hyphenchar\font45}
\DeclareFontShape{U}{mathx}{m}{n}{
      <5> <6> <7> <8> <9> <10>
      <10.95> <12> <14.4> <17.28> <20.74> <24.88>
      mathx10
      }{}
\DeclareSymbolFont{mathx}{U}{mathx}{m}{n}
\DeclareFontSubstitution{U}{mathx}{m}{n}
\DeclareMathAccent{\widecheck}{0}{mathx}{"71}
\DeclareMathAccent{\wideparen}{0}{mathx}{"75}

\usepackage[utf8]{inputenc}
\usepackage[english]{babel}
\usepackage{anysize}
\usepackage{tikz-cd}

\usepackage{lipsum}

\newcommand{\C}{\mathbb{C}}
\newcommand{\N}{\mathbb{N}}
\newcommand{\R}{\mathbb{R}}
\newcommand{\T}{\mathbb{T}}
\newcommand{\Z}{\mathbb{Z}}
\newcommand{\Ad}{\operatorname{Ad}}

\newcommand{\rk}{\operatorname{rk}}

\newcommand{\dist}{\operatorname{dist}}

\newtheorem{teor}{Theorem}[section]
\newtheorem{propop}[teor]{Proposition}
\newtheorem{lemma}[teor]{Lemma}

\newtheorem{definition}[teor]{Definition}
\newtheorem{corol}[teor]{Corollary}
\theoremstyle{remark}

\theoremstyle{remark}
\newtheorem{remark}{Remark}

\input{def.tex}

\title{Concentration of Quasimodes for Perturbed Harmonic Oscillators}

\author[Víctor Arnaiz]{Víctor Arnaiz}

\address{Laboratoire de Mathématiques Jean Leray, Université de Nantes. UMR CNRS 6629, 2
	rue de la Houssinière, 44322 Nantes Cedex 03, FRANCE.}

\email{victor.arnaiz@univ-nantes.fr}

\author[Fabricio Macià]{Fabricio Macià}

\address{M$^2$ASAI, Universidad Politécnica de Madrid.  \newline
ETSI Navales, Avda. de la Memoria, 4. 28040 Madrid, SPAIN.}

\email{fabricio.macia@upm.es}

\begin{document}

\begin{abstract}
In this work, concentration properties of quasimodes for perturbed semiclassical harmonic oscillators are studied. The starting point of this research comes from the fact that, in the presence of resonances between frequencies of the harmonic oscillator and for a generic bounded perturbation, the set of semiclassical measures for quasimodes of sufficiently small width is strictly smaller than the corresponding set of semiclassical measures for the unperturbed system. In this work we precise the description of this set taking into account the Diophantine properties of the vector of frequencies, the separation between clusters of eigenvalues in the spectrum produced by the presence of a perturbation, and the dynamical properties of the classical Hamiltonian flow generated by the average of the symbol of the perturbation by the harmonic oscillator flow. In particular, for the perturbed two-dimensional periodic harmonic oscillator, we characterize the set of semiclassical measures of quasimodes with a width that is critical for this problem. Two of the main ingredients in the proof of these results, which are of independent interest, are: (i) a novel construction of quasimodes based on propagation of coherent states by several commuting flows and (ii) a general quantum Birkhoff normal form for harmonic oscillators.
\end{abstract}

\maketitle

\section{Introduction}

\subsection{The problem and the setting}
\label{motivation}

This article is devoted to the study of eigenfunctions in $L^2(\R^{d})$ of perturbations of Schrödinger operators of the form 
\[
- \frac{1}{2} \hbar^2 \Delta_x + \frac{1}{2} x \cdot Qx,
\] 
with $Q$ a positive definite $d\times d$ matrix with real entries. 
Such an operator is unitarily equivalent to a quantum harmonic oscillator $\widehat{H}_\hbar$:
\begin{equation}
\label{quantum_harmonic_oscillator}
\widehat{H}_\hbar: = \frac{1}{2} \sum_{j=1}^d \omega_j \big( - \hbar^2 \partial_{x_j}^2 +  x_j^2 \big), \quad \omega_j >0,
\end{equation}
where $\omega_1^2\leq \ldots \leq \omega_d^2 $ are the eigenvalues of $Q$.
We will refer to  $\omega = (\omega_1, \ldots , \omega_d) \in \R^d_+$ as the \textit{vector of frequencies}. Recall that the spectrum of $\widehat{H}_\hbar$ in $L^2(\R^d)$ is discrete, since $\widehat{H}_\hbar$ has compact resolvent, and the sequence of eigenvalues  tends to infinity for any fixed $\hbar$. 

Let $  (\varepsilon_\hbar)$ be a sequence of positive real numbers satisfying $\varepsilon_\hbar \to 0$ as $\hbar \to 0^+$, we consider semiclassical perturbations of $\widehat{H}_\hbar$ of the form
\begin{equation}
\label{e:perturbed_quantum_harmonic_oscillator}
\widehat{P}_\hbar := \widehat{H}_\hbar + \varepsilon_\hbar \widehat{V}_\hbar,
\end{equation}
where $(\widehat{V}_\hbar)_{0<\hbar<1}$ is a uniformly bounded family of self-adjoint operators on $L^2(\R^d)$. Notice that, with these assumptions, the spectrum of $\widehat{P}_\hbar$ is still discrete, and remains at distance $O(\varepsilon_\hbar)$ of that of $\widehat{H}_\hbar$.

We aim at studying the asymptotic properties of sequences of solutions $(\Psi_\hbar)$ to the stationary problem
\begin{equation}
\label{e:stationary_problem}
\widehat{P}_\hbar \, \Psi_\hbar = \lambda_\hbar \, \Psi_\hbar, \quad \Vert \Psi_\hbar \Vert_{L^2(\R^d)} = 1,\quad \text{with }\lambda_\hbar \to 1 \text{ as } \hbar \to 0^+,
\end{equation}
where $\widehat{V}_\hbar=\Op_\hbar(V)$ is the semiclassical Weyl quantization of a symbol $V \in \mathcal{C}^\infty(\R^{2d};\R)$ which is bounded together with all its derivatives (that is, $V \in S^0(\R^{2d})$).

This regime is known as the \textit{semiclassical limit}; the quantum to classical correspondence principle in quantum mechanics states that for $\hbar > 0$ small, the behavior of a solution $\Psi_\hbar$ to \eqref{e:stationary_problem}, and more precisely of its \text{position probability density} $\vert \Psi_\hbar \vert^2$, should be related to the dynamics of the underlying classical system. At least in the unperturbed case $\varepsilon_\hbar=0$, the classical system is the Hamiltonian flow of the classical harmonic oscillator; this is the flow 
\[
\phi^H_t\,:\, T^*\R^d \To T^*\R^d,\quad t\in\R
\]
 of the Hamiltonian vector field $X_H:=\sum_{j=1}^d\omega_j(\xi_j\partial_{x_j}-x_j\partial_{\xi_j})$, which is obtained from the classical Hamiltonian
\begin{equation}
\label{classical_harmonic_oscillator}
H(x,\xi) = \frac{1}{2}  \sum_{j=1}^d  \omega_j \big( \xi_j^2 +  x_j^2 \big), \quad (x,\xi) \in T^*\R^d=\R^{d}_x\times\R^d_\xi.
\end{equation}
Using the semiclassical Weyl quantiation, we have that $\widehat{H}_\hbar=\Op_\hbar(H)$. It is natural in this kind of problems to lift to phase-space $T^*\R^d$ all objects of interest, since this is where the classical Hamiltonian flow is defined. Given a normalized function $\psi \in L^2(\R^d)$, the distribution of energy of $\psi$ in the phase space $T^*\R^{d}$ can be described in terms of the associated Wigner function $W^\hbar_{\psi}$ defined by
$$
\mathcal{C}_c^\infty(T^*\R^{d}) \ni a \longmapsto \big \langle \psi , \Op_\hbar(a) \psi \big \rangle_{L^2(\R^d)} =: W^\hbar_{\psi}(a)\in\C,
$$
where $\Op_\hbar(a)$ denotes the semiclassical Weyl quantization of the symbol $a$ and $\langle \cdot, \cdot \rangle_{L^2(\R^d)}$ is the scalar product on $L^2(\R^d)$; background on semiclassical pseudo-differential operators can be found, for instance in \cite{Dim_Sjo99, Zwobook}. The Wigner function $W^\hbar_{\psi}$ is a lift of the position density $|\psi|^2$ in the sense that $\int_{\R^d}W^\hbar_\psi(\cdot,\xi)d\xi=|\psi|^2$. 
		
If $(\Psi_\hbar)_{0<\hbar<1}$ is a sequence of solutions to \eqref{e:stationary_problem} then $(W_{\Psi_\hbar}^\hbar)$ is bounded in the space of tempered distributions $\cS'(T^*\R^d)$. Every accumulation point $\mu$ of this sequence is positive, and therefore defines a positive measure on $T^*\R^d$. In addition, one can prove that $\mu(T^*\R^d)=1$ and, using the calculus of semiclassical pseudodifferential operators, that $\mu$ is concentrated on the level set $H^{-1}(1)$ and is invariant by the classical flow $\phi^H_t$.

In other words, the set $\mathcal{M}(\widehat{P}_\hbar)$ of accumulation points of all possible sequences of Wigner functions of eigenfunctions \eqref{e:stationary_problem} is contained in:
\[
\mathcal{M}(H)=\{\mu \text{ probability measures on }T^*\R^d\,:\, \supp\mu\subseteq H^{-1}(1);\;(\phi^H_t)_*\mu=\mu,\;\forall t\in\R\}.
\]
The elements of $\mathcal{M}(\widehat{P}_\hbar)$ are called \textit{semiclassical measures}. 

The lift property also holds in the limit $\hbar\to 0^+$: If $\mu\in\mathcal{M}(\widehat{P}_\hbar)$ is obtained from a sequence $(\Psi_\hbar)$ satisfying \eqref{e:stationary_problem} and $|\Psi_\hbar|^2 dx\weaksc \nu$, as $\hbar\to 0^+$, in the space of positive measures on $\R^d$ then 
$$
\nu = \int_{\R^d} \mu(\cdot,d\xi).
$$
The measures $\nu$ are sometimes called \textit{quantum limits}.

Our main interest here is characterizing the set $\mathcal{M}(\widehat{P}_\hbar)$, and therefore the set of quantum limits for $\widehat{P}_\hbar$; we will see that the answer to this question depends strongly on the dynamical properties of the Hamiltonian flow generated by $H$, and how these dynamics are influenced by the perturbation. 

The set of semiclassical measures for the non-perturbed harmonic oscillator has been characterized in \cite{ArMa20}. In order to state the result, let
\begin{equation}
\label{e:submodule_start}
\Lambda_\omega := \{ k \in \mathbb{Z}^d \, : \, k \cdot \omega = 0 \}.
\end{equation}
This is a principal submodule of $\Z^d$; denote $d_\omega := d- \rk  \Lambda_\omega$. The quantity $d_\omega$ is crucial in order to understand the dynamics of $\phi^H_t$: an orbit of $\phi^H_t$ issued from $z_0$ is dense in a torus $\cT_\omega(z_0)$ of dimension at most $d_\omega$. In particular, all orbits are periodic when $d_\omega=1$, while most orbits are dense in $d$-dimensional tori when $d_\omega=d$. Proofs of these facts, together with more detailed description of the dynamics of classical harmonic oscillators are presented Section \ref{s:classical}.

One has $\omega \in \Lambda_\omega^\bot$, where $\Lambda_\omega^\perp$ is the orthogonal complement of $\Lambda_\omega$ in $\R^d$. Let $\nu_n\in\Z^d$, $n=1,\hdots,d_\omega$, be a basis of $\Lambda_\omega^\bot$. One can write $\omega=\sum_{n=1}^{d_\omega}v_n\nu_n$ for some real numbers $v_1,\hdots,v_{d_\omega}$. We define, for every $n=1,\hdots,d_\omega$,
\begin{equation}\label{e:decper}
\cH_n(x,\xi):=\frac{1}{2}\sum_{j=1}^d\nu_{n,j} (\xi_j^2+x_j^2), \quad \nu_n = (\nu_{n,1}, \ldots, \nu_{n,d}),
\end{equation}
and note that each flow $\phi^{\cH_n}_t$ is periodic, since $d_{\nu_n}=1$. In particular, if $d_\omega = 1$, then $\cH_1=H$. Write also 
\begin{equation}
    \cH := (\cH_1,\hdots,\cH_{d_\omega}),\quad v:=(v_1,\hdots,v_{d_\omega});
\end{equation}
so that $H=\cH\cdot v$. Finally, put $\Sigma_\cH:=\cH(H^{-1}(1))$

The main result of  \cite{ArMa20} is that a measure $\mu$ belongs to $\cM(\widehat{H}_\hbar)$ (\textit{i.e.} $\mu$ is a semiclassical measure of a sequence of eigenfunctions of the non-perturbed Hamiltonian)  if and only if there exists  $\cE\in\Sigma_\cH$ such that, 
\begin{equation}\label{e:smcho}
\mu \in \cM_{\cE}(\widehat{H}_\hbar):=\{\rho\in  \cM(H)\,:\,\supp\rho\subseteq \cH^{-1}(\cE)\}. 
\end{equation}
When $d_\omega=1$ (periodic case) this shows that $\mathcal{M}(\widehat{H}_\hbar)=\mathcal{M}(H)$; it also shows that $\mathcal{M}(\widehat{H}_\hbar)\neq\mathcal{M}(H)$ as soon as $d_\omega>1$. Notice that the partition of $T^*\R^{d}$ defined by the level sets $\cH^{-1}(\cE)$ does not depend on the particular choice of the basis $\{ \nu_n \, : \, n = 1, \ldots , d_\omega \}$, although the decomposition of $H$ in periodic Hamiltonians $\mathcal{H}_n$ does.

Quantum limits and semiclassical measures can be defined for sequences of eigenfunctions of general elliptic semiclassical operators on a compact Riemannian manifold. In this more general setting, they have been completely characterized in relatively few cases: spheres \cite{Jak_Zel99}, and more generally compact rank-one symmetric spaces \cite{MaciaZoll} or space forms \cite{AM:10}; and the two dimensional torus $\mathbb{T}^2=\R^2/\mathbb{Z}^2$ \cite{Jak97}. These are all examples of completely integrable Hamiltonian systems. Asymptotically vanishing perturbations of those systems, as those defined by \eqref{e:perturbed_quantum_harmonic_oscillator}, have been studied in \cite{Mac09,MaciaRiviere16,Mac_Riv19}, in the case of the sphere and Zoll manifolds, and in \cite{An_Fer_Mac15,  An_Mac14, Bour13,BurqZworski11, MaciaTorus, Mac_Riv18, TothIntegrable, Wunsch2012} when dealing with tori and more general completely integrable systems. 

Not much is known in the case of small, but non-asymptotically vanishing perturbations (KAM systems); however, there have some interesting recent developments in that direction \cite{Ar18,Gomes18,Gom_Hass18}.

In order to understand the structure of the elements in $\mathcal{M}(\widehat{P}_\hbar)$ it is useful to relax the problem, and consider instead semiclassical measures associated to quasimodes. These are approximate solutions to \eqref{e:stationary_problem}; the precise definition is the following.  
\begin{definition}
Let $(\psi_\hbar)$ be a sequence in $L^2(\R^d)$ with $\Vert \psi_\hbar \Vert_{L^2(\R^d)} = 1$. Let $(r_\hbar)$ be a sequence of positive real numbers such that $r_\hbar \to 0$ as $\hbar \to 0$. We say that $(\psi_\hbar)$ is a quasimode for $\widehat{P}_\hbar$ of width $(r_\hbar)$ if, for $0 < \hbar \leq \hbar_0$,
\begin{equation}
\label{e:quasimode_equation}
\widehat{P}_\hbar \, \psi_\hbar = \lambda_\hbar \, \psi_\hbar + R_\hbar,\quad \lambda_\hbar \to 1,\quad \| R_\hbar \|_{L^2(\R^d)}\leq r_\hbar.
\end{equation}
\end{definition} 
Analogously to the case of eigenfunctions, we can consider the set of semiclassical measures associated to quasimodes of width $(r_h)$. We will denote it by $\mathcal{Q}(\widehat{P}_\hbar, r_\hbar)$; one can prove that as soon as $r_\hbar = o(\hbar)$, the same localization and invariance properties as in the eigenstate case hold, and, since eigenstates are clearly quasimodes of any width:
\[
\mathcal{M}(\widehat{P}_\hbar)\subseteq\mathcal{Q}(\widehat{P}_\hbar, r_\hbar)\subseteq \mathcal{M}(H).
\]
The analysis of quasimodes for perturbations of completely integrable systems has a rich literature. The reader can consult,  among many other works, \cite{BabichBuldyrev,CdV:77,LazutkinBook,Pop00I,Pop00II,RalstonQuasimodes, VasyWunsch09, VasyWunsch11, Wunsch2012,ZelditchGB}, and \cite{AnL, Ar20, Ar_Riv18, Hitrik02, HitrikSjostrand04, HitrikSjostrand05, HitrikSjostrand08,HitrikSjostrand08b, HitrikSjostrand12,HitrikSjostrand18,HitrikSjostrandVuNgoc07} for self-adjoint and non self-adjoint perturbations, respectively.

\subsection{Results}
We next state our results concerning the concentration of quasimodes for the perturbed operator $\widehat{P}_\hbar$. In what follows, we will assume that the perturbation is a bounded, self-adjoint semiclassical pseudo-differential operator:
\begin{equation}
\widehat{V}_\hbar = \Op_\hbar(V),\quad V\in S^0(T^*\IR^d),
\end{equation}
where $S^0(T^*\IR^d)$ stands for the set of smooth, real-valued functions on $\IR^d$ that are bounded together with their derivatives of all orders.
 
As we will see below, a major role in our study is played by the average of the symbol $V$ along the orbits of the flow $\phi_t^H$. We define, for any $a \in \mathcal{C}^\infty(T^*\IR^d)$, 
\begin{equation}
\label{average}
\langle a \rangle (z) := \lim_{T \rightarrow \infty} \frac{1}{T} \int_0^T a \circ \phi_t^H (z) dt.
\end{equation}
This limit is well defined and the convergence takes place in the $\mathcal{C}^\infty(T^*\IR^d)$ topology, see Section \ref{s:classical_averages}.  
Our first result states that quasimodes of width $o(\varepsilon_\hbar \hbar)$ produce semiclassical measures that are invariant by two commuting  flows: the flow of the harmonic oscillator $\phi_t^H$ and the Hamiltonian flow $\phi_s^{\langle V \rangle}$ generated by the average $\langle V \rangle $ . 
\begin{teor}
\label{t:invariance_for_quasimodes}
Let $(\psi_\hbar)$ be a quasimode for $\widehat{P}_\hbar$ of width $o(\varepsilon_\hbar \hbar)$. Then any semiclassical measure $\mu$ associated to the sequence $(\psi_\hbar)$ satisfies:
$$
\big( \phi_s^{\langle V \rangle} \circ \phi_t^H \big)_* \mu = \mu ,\quad \forall s,t \in \R.
$$
In other words, if $(r_\hbar)=o(\vareps_\hbar\hbar)$ then
\[
\mathcal{Q}(\widehat{P}_\hbar, r_\hbar)\subseteq \mathcal{M}(H,\la V \ra):=\{\mu\in\cM(H)\,:\, (\phi_s^{\langle V \rangle} )_*\mu=\mu,\, \forall s \in \R\}.
\]
\end{teor}
 Notice that this result is independent of the size of $\vareps_h$, provided that $\varepsilon_\hbar \to $ as $\hbar \to 0^+$. Note also that this result is empty as soon as, in a neighborhood of $H^{-1}(1)$, one has $\langle V\rangle =f\circ \cH$ for some $f\in S^0(\R^{d_\omega})$, in which case the Hamiltonian vector field $X_{\la V \ra}$ evaluated at any $z\in H^{-1}(1)$ gives a vector tangent to the invariant torus $\mathcal{T}_\omega(z)$. As we will see later, this is always the case in the completely non-resonant case $d_\omega = d$. Let $k \geq 0$, and define $\cO_{H^{-1}(1)}(k)$ as the set of those $a \in \mathcal{C}^\infty(\R^{2d})$ such that $\vert a(z) \vert \leq C \dist(z, H^{-1}(1))^k$ for $z$ in a neighborhood of $H^{-1}(1)$. It turns out that:
$$
X_{\langle V \rangle} \vert_{\mathcal{T}_\omega(z)} \in T\mathcal{T}_\omega(z), \quad \forall z \in H^{-1}(1) \iff \exists f\in \cC^\infty(\R^{d_\omega}),\quad \langle V \rangle - f\circ \cH \in\cO_{H^{-1}(1)}(1).
$$
Otherwise, Theorem \ref{t:invariance_for_quasimodes} provides an actual restriction on the set of semiclassical measures of quasimodes.

\begin{corol}\label{c:qmdiffeigenf}
Suppose that $V \in S^0(T^*\R^d)$ is such that its average satisfies $\langle V \rangle - f\circ\cH \not\in \cO_{H^{-1}(1)}(1)$ for every $f\in \cC^\infty(\R^{d_\omega})$, and that $(r_\hbar)=o(\vareps_\hbar\hbar)$. Then
\[
\mathcal{Q}(\widehat{P}_\hbar, r_\hbar)\neq \mathcal{M}(\widehat{H}_\hbar),\quad \mathcal{M}(\widehat{P}_\hbar)\neq \mathcal{M}(\widehat{H}_\hbar).
\]
\end{corol}

In this work, we aim at precising Theorem \ref{t:invariance_for_quasimodes} and Corollary \ref{c:qmdiffeigenf}. 
When $d=2$ and $\phi^H_t$ is periodic, we actually obtain a complete characterization of semiclassical measures of $o(\hbar^2)$ quasimodes. 
\begin{teor}\label{t:main_2d}
Suppose $d=2$,  $d_\omega = 1$ and that $\varepsilon_\hbar = o(\hbar)$. Then there exists a sequence $(r_\hbar)$ of positive real numbers  satisfying $r_\hbar = o(\varepsilon_\hbar \hbar)$ such that $\mu \in \mathcal{Q}(\widehat{P}_\hbar, r_\hbar)$ if and only if there exists $E \in \R$ such that
\begin{equation}\label{e:biloc}
\supp \mu \subset H^{-1}(1) \cap \langle V \rangle^{-1}(E),
\end{equation}
and
$$
\big( \phi_t^H\circ \phi_s^{\langle V \rangle} \big)_* \mu = \mu, \quad \forall t,s \in \R.
$$
\end{teor}
Note, that the restriction on the size of the perturbation in Theorem \ref{t:main_2d} is not necessary in order to obtain the existence of quasimodes whose semiclassical measures satisfy the conclusion, it is only required in order to ensure that all semiclassical measures satisfy the localization property \eqref{e:biloc}. 

Our results are however more general, and address the following questions.
\begin{itemize}
\item[(i)] The threshold $o(\vareps_\hbar \hbar)$ is optimal for Theorem \ref{t:invariance_for_quasimodes}. It is possible to construct many quasimodes of width $O(\vareps_\hbar \hbar)$ for which the conclusion of Theorem \ref{t:invariance_for_quasimodes} fails. A precise description of those quasimodes is given in Theorem \ref{t:quasimodes_1} below.

\item[(ii)] Under suitable dynamical hypotheses on $\phi_s^{\langle V \rangle}$, we prove a converse for Theorem \ref{t:invariance_for_quasimodes}, namely the existence of quasimodes of width $o(\vareps_\hbar \hbar)$ whose semiclassical measures $\mu$ are some prescribed elements of $\mathcal{M}(H,\la V \ra)$. This is stated precisely in Theorem \ref{t:improved_quasimodes} below. 

\item[(iii)] It is possible to give an analogous result to Theorems \ref{t:invariance_for_quasimodes} and \ref{t:improved_quasimodes} when the average of $V$ satisfies $\langle V \rangle = f\circ\cH$. The role of $\langle V \rangle$ is then played by the first non-trivial term $\langle L \rangle$ of the quantum Birkhoff Normal form of $\widehat{P}_\hbar$ (see Appendix \ref{averaging_method}). In this case, the width of the relevant quasimodes is smaller; precise statements are a bit more involved, and we present such a statement in Theorem \ref{t:general_invariance}, Section \ref{s:general}.

\item[(iv)] When $d_\omega = d$, given any $L \in S^0(\R^{2d})$ and $z_0 \in H^{-1}(1)$, the Hamiltonian vector field $X_{\langle L \rangle}(z_0)$ is tangent to the invariant torus issued from the point $z_0$ by the flow $\phi_t^H$. In this case, the set of  semiclassical measures of sequences of quasimodes of width $\varepsilon_\hbar^N$, for arbitrary but finite $N \geq 1$, contains all Haar measures of tori $\mathcal{T} \subset H^{-1}(1)$ that are invariant  by $\phi_t^H$. This case is addressed in Theorem \ref{t:quasimodes_2}; in addition, we show that when the size of the perturbation $\varepsilon_\hbar$ is smaller than the separation between eigenvalues of $\widehat{H}_\hbar$, then semiclassical measures of quasimodes for $\widehat{P}_\hbar$ of width $o(\varepsilon_\hbar \hbar)$ coincide with semiclassical measures of eigenstates of the unperturbed Hamiltonian.

\end{itemize}

Most of our results rely on the construction of a quantum Birkhoff normal form that can be performed only if the frequency vector $\omega$ satisfies a \textit{Diophantine} property. 
\begin{definition}
A vector $\omega \in \R^d_+$ is called partially Diophantine if there exist constants $\varsigma > 0$ and $\gamma \geq 0$ such that
\begin{equation}
\label{e:partially_diophantine}
\vert \omega \cdot k \vert \geq \frac{\varsigma}{\vert k \vert^{\gamma}}, \quad \forall k \in \mathbb{Z}^d \setminus \Lambda_\omega,
\end{equation}
We denote by $\gamma(\omega)$ the infimum of all constants $\gamma \geq  0$ satisfying \eqref{e:partially_diophantine}.
\end{definition}

\begin{remark}
A partially Diophantine vector $\omega$ can be resonant. When $d_\omega = 1$ (this corresponds to a periodic harmonic oscillator), one has that $\omega$ is always partially Diophantine, since it is of the form $\alpha k_0$, with $\alpha > 0$ and $k_0 \in \mathbb{N}^d$. Indeed,  
$$
\vert \omega \cdot k \vert = \alpha \vert k_0 \cdot k \vert \geq \alpha > 0, \quad \forall k \in \mathbb{Z}^d \setminus \Lambda_\omega.
$$
\begin{itemize}
    \item If $d_\omega = d$ (\textit{i.e.} $\omega$ is non-resonant), then \eqref{e:partially_diophantine} means simply that $\omega$ is Diophantine in the usual way.  It is well known that the set of Diophantine vectors has full Lebesgue measure (see for instance \cite{Russ74}). Hence the set of partially Diophantine vectors has also full Lebesgue measure, since it contains the set of Diophantine vectors.

    \item When $d_\omega=1$ then $\gamma(\omega) = 0$. Otherwise, if $d_\omega > 1$ then $\gamma(\omega) \geq d_\omega-1$ (see again \cite{Russ74}).
\end{itemize}
\end{remark}

In the case $d_\omega = d$, and similarly to what happens in KAM theory, this Diophantine exponent is related to the size $\eps_\hbar$ of the perturbation for which the semiclassical measures in $\mathcal{Q}(\widehat{P}_\hbar, r_\hbar)$ for $r_\hbar = o(\eps_\hbar\hbar)$ actually coincide with those of the unperturbed Hamiltonian $\widehat{H}_\hbar$.  
\begin{teor}
\label{t:quasimodes_2}
Suppose that $\omega$ is Diophantine and that $\varepsilon_\hbar = \hbar^\alpha$ for some $\alpha > 0$. Then for every $N > 1 + \alpha$, there exists a width $r_\hbar = O(\hbar^N)$ such that
$$
\mathcal{M}(\widehat{H}_\hbar)\subseteq \mathcal{Q}(\widehat{P}_\hbar, r_\hbar).
$$
If, in addition, $\alpha>1+\gamma(\omega)$ then the inclusion is an equality:
\[
\mathcal{Q}(\widehat{P}_\hbar, r_\hbar)= \mathcal{M}(\widehat{H}_\hbar).
\]
\end{teor}

In all the results that follow we are going to make the following assumption:
\begin{equation}\label{e:ass}\tag{A}
\omega \text{ is partially Diophantine and }\varepsilon_\hbar = O(\hbar^\alpha) \text{ for some } \alpha > 0.
\end{equation}
We next state our first result about the construction of quasimodes of width $O(\varepsilon_\hbar \hbar)$ for $\widehat{P}_\hbar$.

\begin{teor}
\label{t:quasimodes_1}
Suppose \eqref{e:ass} holds. Let $\mu_0 \in \mathcal{M}(\widehat{H}_\hbar)$ with $\supp \mu_0 \subset \langle V \rangle^{-1}(E)$ for some $E \in \R$. Then, for every $T > 0$, and every $\chi_T \in \mathcal{C}_c^\infty(\R)$ supported on $(0,T)$, there exists a quasimode $(\psi_\hbar)$ for $\widehat{P}_\hbar$ of width $r_\hbar = O(\varepsilon_\hbar \hbar)$ with semiclassical measure given by
\begin{equation}
\label{e:evolved_measure}
\mu_T = \frac{1}{ \Vert \chi_T \Vert_{L^2(\R)}^2} \int_\R \chi_T(s)^2 \, \big( \phi_s^{\langle V \rangle} \big)_* \mu_0 \, ds.
\end{equation}
\end{teor}

To go beyond this result and obtain quasimodes of width $o(\varepsilon_\hbar \hbar)$ for $\widehat{P}_\hbar$, we  need to impose some restrictions on the dynamics of the flows $\phi_t^H$ and $\phi_s^{\langle V \rangle}$ along the orbit issued from a point $z_0 \in H^{-1}(1)$.  

As mentioned before, the orbit of $\phi_t^H$ issued from a point $z\in\R^{2d}$ is dense in a torus $\cT_\omega(z)$ of dimension smaller or equal than $d_\omega$. 
This torus carries a unique probability measure $\mu_{\omega,z}$ that is invariant by the flow (see Section \ref{s:classical}).

\begin{definition}
\label{d:birkhoff_average}
Let $L \in S^0(\R^{2d})$, denote by $ \phi_{s}^{\langle L \rangle}$ the Hamiltonian flow of $\langle L \rangle$. We define $\mathcal{I}(H,\langle L \rangle)$ as the set of points $z_0 \in H^{-1}(1)$ such that the Birkhoff average 
\begin{equation}
  \frac{1}{T} \int_0^T \big(\phi_{s}^{\langle L \rangle} \big)_*  \mu_{\omega,z_0}  \, ds  
\end{equation}
converges as $T \to + \infty$ to some measure $\mathcal{A}_{\langle L \rangle}(\mu_{\omega, z_0})$ in the weak-$\star$ topology of Radon probability measures.
\end{definition}

\begin{remark}\label{r:2dinteg}
If $d = 2$, then $H^{-1}(1) = \mathcal{I}(H,\langle L \rangle)$ since the pair $(H,\langle L \rangle)$ defines a completely integrable system.
\end{remark}

\begin{remark}
Since $H^{-1}(1)$ is compact, there always exist points $z_0 \in H^{-1}(1)$ such that $X_{\langle L \rangle}(z_0) \in T_{z_0} \mathcal{T}_\omega(z_0)$. Hence $\mathcal{I}(H,\langle L \rangle) \neq \emptyset$. Moreover, for such points, $\mu_{\omega,z_0} \in \mathcal{M}(\widehat{H}_\hbar)$ is supported on $\mathcal{I}(H, \langle L \rangle)$.  Thus the subset of measures $\mu \in \mathcal{M}(\widehat{H}_\hbar)$ that are supported on $\mathcal{I}(H,\langle L \rangle)$ is also non-empty.
\end{remark}

\begin{remark}
For every Radon measure $\mu\in \mathcal{M}(H)$  supported on $\mathcal{I}(H,\langle L \rangle)$ the measure $\mathcal{A}_{\langle L \rangle}(\mu)$ defined by
\begin{equation}
\label{e:average_measure}
\mathcal{A}_{\langle L \rangle}(\mu) := \lim_{T \to + \infty} \frac{1}{T} \int_0^T \big(\phi_{ s}^{\langle L \rangle} \big)_*  \mu  \, ds,
\end{equation}
(the limit being taken in the weak-$\ast$ topology) is a Radon measure on $H^{-1}(1)$ that is invariant by the two flows $\phi_t^H$ and $\phi_s^{\langle L \rangle}$.
\end{remark}

\begin{teor}
\label{t:improved_quasimodes}
Assume \eqref{e:ass} holds. Then there exists a sequence $(r_\hbar)$ of non-negative real numbers satisfying $r_\hbar = o(\varepsilon_\hbar \hbar)$ such that
$$
\mathcal{Q}(\widehat{P}_\hbar, r_\hbar) \supseteq \big \{ \mathcal{A}_{\langle V \rangle}(\mu) \, : \,  \mu \in \mathcal{M}(\widehat{H}_\hbar) , \quad \supp \mu \subset \mathcal{I}(H,\langle V \rangle) \cap \langle V \rangle^{-1}(E), \quad E \in \R \big \}.
$$
Conversely, if $\varepsilon_\hbar = o(\hbar^{1+\gamma(\omega)})$, then for every sequence $(r_\hbar) \subset \R_+$ satisfying $r_\hbar = o(\varepsilon_\hbar \hbar)$, every $\mu \in \mathcal{Q}(\widehat{P}_\hbar, r_\hbar)$ satisfies: 
$$
\supp \mu \subset \cH^{-1}(\cE) \cap \langle V \rangle^{-1}(E),\quad \big( \phi_t^H \circ \phi_s^{\langle V \rangle} \big)_* \mu = \mu, \quad \forall t,s \in \R.
$$
for some $\cE \in \Sigma_\cH$ and some $E \in \R$.
\end{teor}
\begin{remark}
Theorem \ref{t:main_2d} directly follows from Theorem \ref{t:improved_quasimodes} taking into account Remark \ref{r:2dinteg}.
\end{remark}
Theorems \ref{t:quasimodes_1} and \ref{t:improved_quasimodes} (and ultimately Theorem \ref{t:main_2d}) rely on a construction of quasimodes that extends some of the techniques introduced in \cite{Nonnenmacher17} based on results on propagation of coherent states (as summarized in \cite{Robert12}, see also \cite{Paul_Uribe93}). In \cite{Nonnenmacher17}, a system of adapted coordinates around a hyperbolic periodic orbit is used in order to obtain optimal estimates of the width of quasimodes in that particular setting. The construction presented Section \ref{s:qm} uses a stationary-phase argument that is independent of a particular choice of coordinates, although it does not aim at optimizing widths of the quasimodes obtained. The description of the propagation of coherent and excited states by the quantum multiflow of \cite{Robert12} and the properties of these states with respect to the Wigner transform are sufficient to compute the Wigner function of our desired quasimode in a rather flexible way. This allows us to reach any multiorbit having well defined Birkhoff average, in the sense of Definition \ref{d:birkhoff_average}, and not necessarily being periodic or quasiperiodic.

The other ingredient that comes into play is a Birkhoff normal form form for $\widehat{P}_\hbar$, (Proposition \ref{averaging-lemma} of Appendix \ref{averaging_method}) that extends to the case of general vectors of frequencies those given in \cite{CharVNgoc,Sj92} (see also \cite{Gui_Ur_Wa12} for closely related results). Refinements of these normal forms have been given in \cite{Gui_Ur_Wa15,Hall_Hit_Sjo15, Le_Floch14} for certain classes of Hamiltonians. In the particular case $d=2$ and $d_\omega = 1$, Theorem \ref{t:main_2d} covers more general situations than those that may in principle be reached, for instance, by the normal form given in \cite[Theorem 4.1]{Gui_Ur_Wa15}. In particular, if the reduction of $\langle V \rangle$ by the periodic flow is a Morse function (not necessarily perfect), then our result describes the semiclassical measures of quasimodes of width $o(\varepsilon_\hbar \hbar)$ concentrating on periodic orbits of the harmonic oscillator which are saddle points for $\langle V \rangle$, hence hyperbolic points for its reduction by the periodic flow (see also \cite{Le_Floch14} for Bohr-Sommerfeld conditions on hyperbolic points of a compact Riemannian surface).  On the other hand, notice that no invariant probability measure can be supported on homoclinic or heteroclinic orbits of the flow $\phi_s^{\langle V \rangle}$ connecting saddle points outside these limit points.

\subsection{Structure of the article} 
The necessary elements of the classical dynamics of Harmonic Oscillators are recalled in Section \ref{s:classical}. Section \ref{s:general} presents the proof of Theorem \ref{t:invariance_for_quasimodes} as well as a generalization, Theorem \ref{t:general_invariance}.  In Section \ref{s:sep} we prove Theorem \ref{t:quasimodes_2} as well as a localization result for quasimodes in the Diophantine case, Proposition \ref{p:localization}. Section \ref{s:cs} describes the necessary tools of the theory of propagation of coherent and excited states that are needed in the construction of quasimodes in Section \ref{s:qm}. The proofs of Theorems \ref{t:quasimodes_1} and \ref{t:improved_quasimodes} is presented in Section \ref{s:proofs}, as well as that of Corollary \ref{c:corol_general_invariance}, a partial converse of Theorem \ref{t:general_invariance}. Finally Appendix \ref{averaging_method} presents the construction of a Quantum Birkhoff Normal Form, which is of independent interest, that is used in the proofs of several of the aforementioned results.

\subsection*{Acknowledgments}
The authors would like to warmly thank Stéphane Nonnenmacher and Gabriel Rivière for their useful comments on an early version of the manuscript.
This research have been supported by MTM2017-85934-C3-3-P (MINECO, Spain). VA has been supported by a predoctoral grant from Fundación La Caixa - Severo Ochoa International Ph.D. Program at the Ins\-tituto de Ciencias Matemáticas (ICMAT-CSIC-UAM-UC3M-UCM), by the European Research Council (ERC) under the European Union Horizon 2020 research and innovation programme (grant agreement No. 725967), and by ANR project Aléatoire, Dynamique et Spectre.

\section{Classical Averaging}\label{s:classical}

In this section we discuss several useful properties of classical and quantum averages, that are be used in the article.

\subsection{Review on the classical dynamics of harmonic oscillators}\label{ss:classicalD}
The classical harmonic oscillator forms a completely integrable dynamical system. The Hamiltonian $H$ can be written as a function of $H_1, \ldots, H_d$, defined by
\begin{equation}\label{e:hj}
H_j(x,\xi) := \frac{1}{2} \big( \xi_j^2 + x_j^2 \big),
\end{equation}
by putting
\begin{equation}
\label{e:linear_for_H}
H = \mathcal{L}_\omega(H_1, \ldots, H_d),
\end{equation}
where $\mathcal{L}_\omega : \R^d_+ \to \R$ is the linear form defined by $\mathcal{L}_\omega(E) = E \cdot \omega$, and $\{ H_j , H_k \} = 0$ for every $j,k\in \{1, \ldots , d \}$.  As a consequence, $\phi_t^H$, the Hamiltonian flow of $H$, can be written as 
\[
\phi_t^H(z)=\Phi_{t\omega}(z), \quad t\in\R, \;z = (x,\xi)\in\R^{2d},
\] 
where
\begin{equation}
\label{e:multiflow}
\Phi_\tau(z) := \phi_{t_d}^{H_d} \circ \cdots \circ \phi_{t_1}^{H_1}(z), \quad \tau=(t_1, \ldots, t_d) \in \R^d,
\end{equation}
and $\phi_{t}^{H_j}$ denotes the flow of $H_j$. These flows are totally explicit, they act as a rotation of angle $t$ on the plane $(x_j,\xi_j)$. If one identifies points $(x_j,\xi_j)$ in that plane to the complex numbers $z_j:=x_j+i\xi_j$, then $\phi_{t}^{H_j}$ acts on that plane as $e^{-it}z_j$ and fixes the points in its orthogonal complement. Therefore, $\tau\mapsto\Phi_\tau(z)$ is $2\pi\Z^d$-periodic for every $z\in\R^{2d}$
and we will identify it to a function defined on the torus $\mathbb{T}^d:=\mathbb{R}^d/2\pi\mathbb{Z}^d$. 

Let
\begin{equation}
\label{e:various_notation}
\mathbb{M}:=(H_1,\hdots,H_d),\quad X:=(0,\infty)^d,\quad \Sigma:=\R^d_+\setminus X.
\end{equation}
For every $E\in\R^d_+$, let $\mathcal{T}_E:=\mathbb{M}^{-1}(E)$; these sets are invariant by the flow $\phi^H_t$. If $E\in X$ then $\mathcal{T}_E$ is Lagrangian and, for every $z_0\in \mathbb{M}^{-1}(E)$,
\[
\T^d\ni\tau \longmapsto \Phi_\tau(z_0)\in\cT_E
\]
is a diffeomorphism; moreover, 
\[
\phi^H_t\circ \Phi_\tau(z_0) = \Phi_{\tau +t\omega}(z_0),\quad \forall t\in\R.
\]
Therefore, Konecker's theorem then shows that the orbit of $\phi^H_t$ of any point $z_0\in \mathbb{M}^{-1}(X)$ is dense in a $d_\omega$-dimensional subtorus $\cT_\omega(z_0)$ of $\cT_E$.
To deal with the general case, define for $v\in\R^d$ and $E\in\R^{d}_+$, 
\begin{equation}
\label{e:projection_degeneracies}
\pi_E(v):=(\mathbf{1}_{(0,\infty)}(E_1)v_1,\hdots,\mathbf{1}_{(0,\infty)}(E_d)v_d).
\end{equation}
When $\mathbb{M}(z_0)=E\in\Sigma$, the map $\Phi_{z_0}$ is no longer a diffeomorphism but one still has:
\[
\phi^H_t\circ \Phi_\tau(z_0) = \Phi_{\tau +t\pi_{E}(\omega)}(z_0),\quad \forall t\in\R.
\]
Therefore, the orbit issued from such $z_0$ is again dense in a subtorus of $\cT_E$ of dimension lower than $d_\omega$, which we will still denote by $\cT_\omega(z_0)$. 

\subsection{Classical averages}
\label{s:classical_averages}
Let $\nu_n\in\IZ^d$, $n=1,\hdots,d_\omega$ be an orthogonal basis of 
\[
\Lambda_\omega^\bot = \{k\in\Z^d\,:\,\omega\cdot k =0\}^\bot 
\]
consisting of primitive vectors\footnote{A vector in $\Z^d$ is \textit{primitive} provided its components are relatively prime integers.}.
Write $\omega =v_1\nu_1+\hdots+v_{d_\omega}\nu_{d_\omega}$; one can check that $d_\omega$ coincides with the dimension of the linear subspace of $\IR$ over $\IQ$ spanned by the components of $\omega$ and that $v_1,\hdots,v_{d_\omega}$ form a basis of this space. This implies in particular that all $v_n$, $n=1,\hdots,d_\omega$ are different from zero. Denote by $\rho_\omega : \R^{d_\omega}\To \R^d$ the linear map that sends the canonical basis of $\R^{d_\omega}$ onto the basis $\{ \nu_1,\hdots, \nu_{d_\omega} \}$. Clearly, setting $v=(v_1,\hdots,v_{d_\omega})$,
\[
\rho_\omega\,:\,\R^{d_\omega}\To \Lambda_\omega^\bot \;\text{ is a linear isomorphism and } \; \rho_\omega(v)=\omega.
\]
The fact that the basis we have chosen consists of primitive vectors implies that $\rho_\omega^{-1}(\Lambda_\omega^\bot\cap \IZ^d)=\IZ^{d_\omega}.
$
Note also that 
\begin{equation}\label{e:v_nonresonant}
    k\cdot v \neq 0,\quad \forall k\in\IZ^{d_\omega}\setminus\{0\}; 
\end{equation}
this follows from $k\cdot v =\tilde{k}\cdot \omega$ where $\tilde{k}:=\sum_{n=1}^{d_\omega} |\nu_n|^{-2} k_n \nu_n$. Since $\tilde{k}\in\IQ^d$, one cannot have $\tilde{k}\cdot \omega = 0$ unless $k=0$. If, in addition, $\omega$ is partially Diophantine \eqref{e:partially_diophantine}, then $v$ is Diophantine: for every $k\in\IZ^{d_\omega}\setminus\{0\}$,
\begin{equation}
\label{e:partially_diophantine2}
|k\cdot v|=\vert \tilde{k} \cdot \omega \vert \geq \frac{\varsigma'}{\vert k \vert^{\gamma}}, \quad \varsigma' := \varsigma\prod_{n=1}^{d_\omega}|\nu_n|^{-2\gamma}.
\end{equation}
Define, for $\tau\in\R^{d_\omega}$, 
\[
\Phi^\cH_\tau := \phi^{\cH_1}_{\tau_1}\circ\cdots\circ\phi^{\cH_{d_\omega}}_{\tau_{d_\omega}}.
\]
One has $\Phi^\cH_\tau=\Phi_{\rho_\omega(\tau)}$ and $\phi^H_t=\Phi^\cH_{tv}$; therefore $\tau\mapsto\Phi^\cH_\tau$ is $2\pi\IZ^{d_\omega}$-periodic and can be identified to a map on $\T^{d_\omega} := \rho_\omega^{-1}(\T_\omega) = \IR^{d_\omega}/2\pi \mathbb{Z}^{d_\omega}$.

Given any function $a \in \mathcal{C}^\infty(\R^{2d})$,  we define its average $\langle a \rangle$ along the flow $\phi_t^H$ as
\begin{equation}
\label{average-definition1}
\langle a \rangle := \lim_{T \to \infty} \frac{1}{T} \int_0^T a \circ \phi_t^H dt = \lim_{T\to\infty}\frac{1}{T}\int_0^T a\circ\Phi^\cH_{tv}dt.
\end{equation}
\begin{propop}\label{p:av_prop}
The limit \eqref{average-definition1} is well defined in the $\mathcal{C}^\infty(\R^{2d})$ topology and
\begin{equation}
\label{e:different_formulas_average}
\langle a \rangle = \frac{1}{(2\pi)^{d_\omega}}  \int_{\T^{d_\omega}} a \circ \Phi_\tau^\cH d\tau.
\end{equation}
In addition, for every $a\in\mathcal{C}^{\infty}(\mathbb{R}^{2d})$ the following hold:
\begin{itemize}
    \item[(i)]  If $d_\omega = d$ then there exists $f\in\mathcal{C}^\infty(\R^{d})$ such that $\langle a \rangle = f\circ\mathbb{M}$.
    \item[(ii)] $\{\langle a \rangle,\langle b \rangle\}=0$ for every $b\in \mathcal{C}^{\infty}(\mathbb{R}^{2d})$ if and only if $f\in\mathcal{C}^\infty(\R^{d_\omega})$ exists such that
    \begin{equation}
    \label{e:function_classic}
    \langle a \rangle = f\circ\cH.
    \end{equation}
\end{itemize}
\end{propop}
\begin{proof}
To see this, write $a \in \mathcal{C}^\infty(\R^{2d})$ as a Fourier series as follows. First define
\begin{equation}
\label{e:fourier_coefficients_oscillator}
a_k := \int_{\T^{d_\omega}} a \circ \Phi^\cH_\tau e^{-ik\cdot \tau}d\tau,\quad \forall k\in \IZ^{d_\omega}.
\end{equation}
Since the function $(\tau,z)\mapsto a \circ \Phi^{\mathcal{H}}_{\tau}(z)$ is smooth on $\T^{d_\omega}\times \R^{2d}$  it follows that the Fourier series
\[
a\circ\Phi^\cH_\tau =\frac{1}{(2\pi)^{d_\omega}}\sum_{k\in\IZ^{d_\omega}}a_k e^{ik\cdot\tau}
\]
converges absolutely in $\cC^\infty(K)$ for every compact set $K\subset\R^{2d}$.  Identity \eqref{e:v_nonresonant} and Kronecker's theorem then imply that the average $\langle a \rangle$ is given by 
\begin{equation*}
\langle a \rangle = \lim_{T\to\infty}\frac{1}{T}\int_0^T \frac{1}{(2\pi)^{d_\omega}}\sum_{k\in\IZ^{d_\omega}} a_k e^{itk\cdot v}dt=\frac{1}{(2\pi)^{d_\omega}}  \int_{\T^{d_\omega}} a \circ \Phi_\tau^\cH d\tau,
\end{equation*}
which concludes the proof of \eqref{e:different_formulas_average}.
Let us now prove (ii). The converse implication is immediate, let us focus on the direct one. Consider a linear isomorphism $\rho'_\omega:\R^{d-d_\omega}\To\langle \Lambda_\omega\rangle$ and define $\varrho:\IR^d=\IR^{d_\omega}\oplus\IR^{d-d_\omega}\To \Lambda_\omega^\bot\oplus\langle \Lambda_\omega\rangle=\IR^d$ by $\varrho(\tau, \tau') = \rho_\omega(\tau)+\rho'_\omega(\tau')$. Note that,
\begin{align*}
&\varrho^{*}\circ \mathbb{M}  = (\cH, \mathcal{H}'),
\end{align*}
where $\varrho^{*}:\R^{d}\To\R^{d_\omega}\oplus\R^{d-d_\omega}$ denotes the dual map. The hypothesis implies that $\langle a \rangle$ commutes with all the $H_j$, $j=1,\hdots, d$, which in turn gives:
\begin{align*}
\langle a \rangle = \frac{1}{(2\pi)^{d}} \int_{\T^{d}} \langle a \rangle \circ \Phi_\tau d\tau 
  = \tilde{f} \circ \mathbb{M},
\end{align*}
where $\tilde{f}\in\cC^\infty(\IR^d)$ is constructed as follows: consider the multiradial part of $a$:
\[
a_{\mathrm{rad}}(r_1,\hdots,r_d):=\frac{1}{(2\pi)^{d}} \int_{\T^{d}} \tilde{a} \big( r_1 e^{-i t_1}, \ldots, r_d e^{-it_d} \big) dt_1 \cdots dt_d
\]
where $\tilde{a}$ is obtained from $a$ after identifying $(x,\xi) \equiv z = x + i \xi$. A theorem of Whitney \cite{Whit43} then shows that $a_{\mathrm{rad}}(\sqrt{2E_1},\dots,\sqrt{2E_d})=\tilde{f}(E_1,\hdots,E_d)$ for some $\tilde{f}\in\cC^\infty(\IR^d)$ which has the claimed property.
Define now 
\begin{equation}
\label{e:fdeH}
f := \tilde{f} \circ (\varrho^*)^{-1}.
\end{equation}
Note that the map $(x,\xi)\mapsto (\cH(x,\xi),\cH'(x,\xi))$ has rank $d$ for $(x,\xi)\in H^{-1}(X)$; Liouville's theorem then implies the existence, around any point $(x_0,\xi_0)\in H^{-1}(X) $, of locally defined somooth functions $b_1,\hdots,b_{d-d_\omega}$ such that for every $i=1,\hdots,d-d_\omega$ and $(x,\xi)$ in a neighborhood of $(x_0,\xi_0)$,
\[
\{\cH_j,b_i\}(x,\xi) =\{\cH_k',b_i\}(x,\xi)=0,\quad \{\cH_i',b_i\}(x,\xi)=1,
\]
for $j=1,\hdots,d_\omega$ and $k=1,\hdots,d-d_\omega$ with $i\neq k$. 
Since for any $b \in \cC^\infty(\R^{2d})$,
\[
0=\{\langle a \rangle,\langle b \rangle \} = \sum_{j=1}^{d-d_\omega}\partial_{\cH'_j}f(\cH,\cH')\{\cH'_j,\langle b \rangle\},
\]
one finds, setting $\langle b \rangle = b_i$ that actually $f$ does not depend on $\cH'$. Finally (i) follows from (ii) since when $d_\omega = d$, the function $\langle a \rangle$ commutes with all the $H_j$ for $j = 1, \ldots, d$.
\end{proof}

\section{Quantum Averaging and Classical Invariance}\label{s:general}

We now present the proof of Theorem \ref{t:invariance_for_quasimodes} and an example of generalization, Theorem \ref{t:general_invariance}, via the quantum Birkhoff normal form given in Appendix \ref{averaging_method}. Assuming that $\omega$ is partially Diophantine, one can conjugate the operator $\widehat{P}_\hbar$ by some suitable unitary operator so that the perturbation $\varepsilon_\hbar \widehat{V}_\hbar$ is averaged by the quantum flow $e^{-i\frac{t}{\hbar}\widehat{H}_\hbar}$ up to order $\varepsilon_\hbar^N$ for arbitrary large $N$. However, the first order approximation can be constructed without making any assumption on the frequency vector. Let $T > 0$ and let $a \in S^0(\R^{2d})$, we define its quantum average $\langle \Op_\hbar(a) \rangle_T$ at time $T$ by
\begin{equation}
\label{e:quantum_averageT}
\langle \Op_\hbar(a) \rangle_T :=\frac{1}{T} \int_0^T e^{-i\frac{t}{\hbar}\widehat{H}_\hbar} \, \Op_\hbar(a) \, e^{i\frac{t}{\hbar}\widehat{H}_\hbar} dt.
\end{equation}
\begin{propop}
\label{p:quantum_average}
The operators $\langle \Op_\hbar(a) \rangle_T$ converge in the strong operator  $\mathcal{L}(L^2)$-norm as $T\to\infty$ to an operator $\langle \Op_\hbar(a) \rangle$ that satisfies
\begin{equation}
\label{e:quantum_average}
\langle \Op_\hbar(a) \rangle = \Op_\hbar(\langle a \rangle).
\end{equation}
In addition, for every $a\in S^{0}(\mathbb{R}^{2d})$ the following hold:
\begin{itemize}
    \item[(i)] If $d_\omega = d$ then there exists $f_\hbar\in\mathcal{C}^\infty(\R^{d})$ such that 
    \[
    \langle \Op_\hbar(a) \rangle = f_\hbar(\Op_\hbar({H}_{1}), \ldots, \Op_\hbar({H}_{d})).
    \]
    \item[(ii)] $[\langle \Op_\hbar(a) \rangle,\langle \Op_\hbar(b) \rangle]=0$ for every $b\in S^{0}(\mathbb{R}^{2d})$ if and only if $f_\hbar\in S^0(\R^{d_\omega})$ exists such that
    \[
    \langle \Op_\hbar(a) \rangle = f_\hbar(\Op_\hbar(\cH_1), \ldots, \Op_\hbar(\cH_{d_\omega})).
    \]
    Moreover, 
    \begin{equation}\label{e:sym_compos}
    f_\hbar(\Op_\hbar(\cH_1), \ldots, \Op_\hbar(\cH_{d_\omega}))=\Op_\hbar(f\circ \cH)+ O_{\cL(L^2(\R^d))}(h),
    \end{equation}
    where $f$ is given by \eqref{e:function_classic}.
\end{itemize}
\end{propop}
\begin{proof}
The existence of quantum averages and formula \eqref{e:quantum_average} follow from Proposition \ref{p:av_prop}, the Calderón-Vaillancourt theorem and Egorov's theorem, which is exact since $\widehat{H}_\hbar$ has quadratic symbol.

The non-trivial implication of (ii) is proved as follows. The hypothesis implies that 
\begin{equation}\label{e:commHj}
[\langle \Op_\hbar(a) \rangle,\Op_\hbar({H}_{j})]=0,\quad  j=1,\hdots,d.
\end{equation}
One has that the spectra $\Spec(\Op_\hbar({H}_{j}))=\hbar(\IZ_+ + 1/2)$ are simple. Joint eigenfunctions of $(\Op_\hbar({H}_{j}))_{j=1,\hdots,d}$ are precisely the Hermite functions, which form an orthonormal basis of $L^2(\R^d)$. This ensures that 
\[
\Spec(\langle \Op_\hbar(a) \rangle )=\left\{\tilde{f}_\hbar\left(\hbar k_1+\frac{\hbar}{2},\hdots, \hbar k_d+\frac{\hbar}{2} \right)\,:\, k\in\IZ^d_+\right\}, \text{ for some } \tilde{f}_\hbar\in\cC^\infty(\R^d).
\]
The function $f_\hbar$ defined from $\tilde{f}_\hbar$ by \eqref{e:fdeH} has the desired properties. Moreover, using the statement (ii) of Proposition \ref{p:av_prop} and the functional calculus for semiclassical pseudodifferential operators, we obtain \eqref{e:sym_compos} where $f$ is given by \eqref{e:function_classic}. Finally, the proof of (i) follows from (ii) and the fact that \eqref{e:commHj} holds in this case, because of \eqref{e:quantum_average}.
\end{proof}

\begin{proof}[Proof of Theorem \ref{t:invariance_for_quasimodes}]

Let $(\psi_\hbar)$ be a quasimode for $\widehat{P}_\hbar$ of width $(r_\hbar)=o(\epsilon_\hbar \hbar)$ with semiclassical measure equal to $\mu$.
Given $a \in \mathcal{C}_c^\infty(\R^{2d})$,  denote by $\widehat{A}_\hbar=\langle \Op_\hbar(a) \rangle$ the quantum average of $\Op_\hbar(a)$, as given by \eqref{e:quantum_average}. One has
\begin{equation}
\label{wigner_equation}
o(1) = \frac{i}{\varepsilon_\hbar\hbar} \left \langle  \psi_\hbar,  [\widehat{H}_\hbar + \varepsilon_\hbar \widehat{V}_\hbar,  \widehat{A}_\hbar  ]\psi_\hbar \right \rangle_{L^2(\R^d)}.
\end{equation}
Noting that
$$
[\widehat{H}_\hbar, \widehat{A}_\hbar] = \frac{\hbar}{i} \Op_\hbar( \{ H, \langle a \rangle \}) = 0,
$$
we can use the commutator rule for Weyl pseudodifferential operators to obtain
\begin{equation}
\label{conmutator-ruleG}
 \frac{i }{\varepsilon_\hbar\hbar}  [\widehat{H}_\hbar + \varepsilon_\hbar \widehat{V}_\hbar,  \widehat{A}_\hbar  ] = \Op_\hbar(\{ V, \langle a \rangle \}) + O(  \hbar^2).
\end{equation}
Inserting this identity in (\ref{wigner_equation}), letting $\hbar \rightarrow 0^+$ and using that $\mu$ is invariant by the flow $\phi_t^H$ we obtain that
\begin{equation}
\label{integral-wigner-equation}
0 =  \int_{\R^{2d}} \{ V, \langle a \rangle \}(x,\xi)\mu(dx,d\xi)=\int_{\R^{2d}} \{ \langle V \rangle, \langle a \rangle \}(x,\xi)\mu(dx,d\xi),\quad \forall a \in \mathcal{C}_c^\infty(\R^{2d}).
\end{equation}
Since $\langle a \circ \phi_{t}^{\langle V \rangle} \rangle = \langle a \rangle \circ \phi_{t}^{\langle V \rangle}$, we can use \eqref{integral-wigner-equation} and the invariance under the flow of the Harmonic Oscillator to conclude 
\begin{align*}
\frac{d}{dt} \int_{\R^{2d}}  a \circ \phi_{t}^{\langle V \rangle}  (x,\xi)\mu(dx,d\xi)&=\frac{d}{dt} \int_{\R^{2d}} \langle a \circ \phi_{t}^{\langle V \rangle} \rangle (x,\xi)\mu(dx,d\xi) \\ &= \int_{\R^{2d}} \{ \langle V \rangle, \langle a \rangle \}\circ \phi_{t}^{\langle V \rangle}(x,\xi)\mu(dx,d\xi) = 0.
\end{align*}
Therefore,
$$
\int_{\R^{2d}}  a(x,\xi)\mu(dx,d\xi) = \int_{\R^{2d}}  a \circ \phi_{t}^{\langle V \rangle}  (x,\xi)\mu(dx,d\xi), \quad \forall a \in \mathcal{C}_c^\infty(\R^{2d}).
$$
The density of $\mathcal{C}_c^\infty(\R^{2d})$ in $\mathcal{C}_c(\R^{2d})$ completes the proof the theorem. 
\end{proof}

As mentioned in the introduction, Theorem \ref{t:invariance_for_quasimodes} gives no additional information on the structure of semiclassical measures when the average $\langle V \rangle$ is a function of the periodic Hamiltonians $\cH_1,\hdots,\cH_{d_\omega}$. In this case, one can still get additional invariance for semiclassical measures associated to quasimodes of smaller width. The relevant flow is determined by the principal symbol of the first non-vanishing term in the quantum Birkhoff normal form described in Proposition \ref{averaging-lemma}.

Our next result is an example of such a situation. We have chosen to state it in lesser generality, in particular letting $\varepsilon_\hbar=\hbar^k$ with $k$ an integer, in order to keep the statement relatively simple. We refer the reader to \cite{Arnaiz_tesis} for further generalizations of this type of result, in particular to the case of the time-dependent Schrödinger equation.

\begin{teor}\label{t:general_invariance}
Suppose that $\varepsilon_\hbar=\hbar^k$ with $k\in\N$ and that $\omega$ is partially Diophantine. Suppose $j,N\in\N$ exist such that the terms in formula \eqref{quantum_normal_form} satisfy 
\[
\langle \widehat{R}_{i,\hbar} \rangle = f_i(\Op_\hbar(\cH_1),\hdots, \Op_\hbar(\cH_{d_\omega})), \quad f_i\in S^0(\R^{d_\omega}),\quad i=1,\hdots, j-1,
\]
and 
\[
\langle \widehat{R}_{j,\hbar} \rangle =\Op_\hbar(\langle L \rangle) + \hbar\Op_\hbar(\langle r_j \rangle)+ O_{\cL(L^2)}(\hbar^2)
\] 
such that $\langle L \rangle$ is not a function of $\cH_1,\hdots, \cH_{d_\omega}$.

Then any semiclassical measure $\mu$ associated to a quasimode $(\psi_\hbar)$ for $\widehat{P}_\hbar$ of width $o(\hbar^{kj+1})$ satisfies:
$$
\big( \phi_s^{\langle L \rangle} \circ \phi_t^H \big)_* \mu = \mu ,\quad \forall s,t \in \R.
$$
\end{teor}

\begin{remark}
If additionally, if $kj > 1 + \gamma(\omega)$ and $H^{-1}(1) = \mathcal{I}(H,\langle L \rangle)$, then it is possible to prove a converse for Theorem \ref{t:general_invariance}, see Corollary \ref{c:corol_general_invariance} 
\end{remark}

\begin{proof}
The hypotheses of the theorem and formula \eqref{quantum_normal_form} imply that:
\begin{align}\label{e:conj}
   U_{N,\hbar}^* \big( \widehat{H}_\hbar + \varepsilon_\hbar \widehat{V}_\hbar \big) U_{N,\hbar} & \\[0.2cm]
    & \hspace{-3cm} =  F_h(\Op_\hbar(\cH_1),\hdots, \Op_\hbar(\cH_{d_\omega}))+\hbar^{kj}\langle \Op_{\hbar}(L) \rangle + \hbar^{kj+1}\langle \Op_{\hbar}(r_{j}) \rangle+ O_{\mathcal{L}(L^2)}(\hbar^{kj+2}), \notag
\end{align}
where
\[
F_h(\Op_\hbar(\cH_1),\hdots, \Op_\hbar(\cH_{d_\omega})):=\widehat{H}_\hbar + \sum_{i=1}^{j-1}\hbar^{ki}f_i(\Op_\hbar(\cH_1),\hdots, \Op_\hbar(\cH_{d_\omega}))
\]
Define $\widehat{A}_\hbar:=U_{N,\hbar}\langle\Op_\hbar(a)\rangle U_{N,\hbar}^*$; then \eqref{e:conj} together with the fact that a quantum average commutes with every operator of the form $F_h(\Op_\hbar(\cH_1),\hdots, \Op_\hbar(\cH_{d_\omega}))$, gives
\begin{align*}
    o(1) 
    =& \frac{i}{\hbar^{kj+1}} \left \langle  \psi_\hbar,  [\widehat{H}_\hbar + \hbar^k \widehat{V}_\hbar,  \widehat{A}_\hbar]\psi_\hbar \right \rangle_{L^2(\R^d)}\\
    =&\frac{i}{\hbar^{kj+1}} \left \langle  U_{N,\hbar}^*\psi_\hbar,  [ \hbar^{kj} \langle \Op_\hbar(L)\rangle+\hbar^{kj+1} \langle \Op_\hbar(r_{j})\rangle,  \langle\Op_\hbar(a)\rangle]U_{N,\hbar}^*\psi_\hbar \right \rangle_{L^2(\R^d)}+O(\hbar).
\end{align*}
Using the symbolic calculus and the Calderón Vaillancourt theorem, we infer:
\begin{equation}\label{e:conj2}
    \left \langle  U_{N,\hbar}^*\psi_\hbar,  \Op_\hbar(\{\langle L \rangle,  \langle a \rangle\})U_{N,\hbar}^*\psi_\hbar \right \rangle_{L^2(\R^d)}=o(1).
\end{equation}
On the other hand, \eqref{FIO-estimate} implies that:
\begin{equation*}
    \big \langle  \psi_\hbar, \Op_\hbar(a)  \psi_\hbar \big \rangle_{L^2(\R^d)} = 
    \big \langle   U_{N,\hbar}^* \psi_\hbar  ,  \Op_\hbar(a) U_{N,\hbar}^* \,\psi_\hbar \big \rangle_{L^2(\R^d)} + O(\hbar^k). 
\end{equation*}
In particular $\mu$ is also the semiclassical measure of the sequence $(U_{N,\hbar}^* \,\psi_\hbar)$. Taking limits $\hbar\to 0^+$ in \eqref{e:conj2}, we deduce that:
\begin{equation*}
    \int_{\R^{2d}} \{ \langle L \rangle, \langle a \rangle \}(x,\xi)\mu(dx,d\xi)=0,\quad \forall a \in \cC^\infty_c(\R^{2d}).
\end{equation*}
From this point, one concludes following the same lines of the proof of Theorem \ref{t:invariance_for_quasimodes}.  
\end{proof}

\begin{remark}
\label{r:simpler_form}
In the case $\omega = (1,\ldots ,1)$, one has $\langle \widehat{R}_{2,\hbar} \rangle = \Op_\hbar(\langle L\rangle) + O(\hbar)$ where
\begin{equation}
\label{e:simpler_form}
\langle L \rangle = \frac{1}{4\pi} \int_0^{2\pi} \int_0^t \{ V \circ \phi_s^H , V \circ \phi_t^H \} ds \, dt.
\end{equation}
In general, as soon as $d_\omega < d$, it is not difficult to find examples of perturbations $V \in S^0(\R^{2d})$ for which $\langle V \rangle$ is constant but the principal symbol of $\langle \widehat{R}_{1,\hbar} \rangle$ is not a function of the periodic Hamiltonians $\cH_j$.
\end{remark}

\section{Spectral Separation and Classical Localization}\label{s:sep}

The Diophantine-type condition \eqref{e:partially_diophantine} is related to the speed of convergence of 
$$
\frac{1}{T} \int_0^T a \circ \phi_t^H dt
$$
to the average $\langle a \rangle$ as $T \to \infty$ and, as it is well known, this is crucial when dealing with the classic problem of \textit{small denominators} in KAM theory. 

On the quantum side, \eqref{e:partially_diophantine} implies that there exists a minimal spacing between consecutive eigenvalues of $\widehat{H}_\hbar$. The spectrum of $\widehat{H}_\hbar$ on $L^2(\R^d)$ is:
\begin{equation}
\label{e:spectrum}
\operatorname{Sp}\big( \widehat{H}_\hbar \big) = \left \{ \lambda_{k,\hbar}= \hbar \omega\cdot k+\frac{\hbar|\omega|_1}{2}\,:\, k \in\Z^d_+ \right \},
\end{equation}
and therefore, for every $\delta\in(0,1)$ there exists constants $C,\hbar_0>0$ such that, for every $0<\hbar<\hbar_0$,
\begin{equation}\label{e:spacing}
\min\{|\lambda_\hbar-\lambda'_\hbar| \,:\, \lambda_\hbar,\lambda'_\hbar\in\Spec(\widehat{H}_\hbar)\cap (1-\delta,1+\delta),\; \lambda_\hbar\neq\lambda'_\hbar \}\geq C\hbar^{1+\gamma(\omega)}.
\end{equation}
In particular, when the size of the perturbation $(\varepsilon_\hbar)$ is $o(\hbar^{1+\gamma(\omega)})$, the spectrum of $\widehat{P}_\hbar$ consists of clusters of eigenvalues centered around the eigenvalues of the unperturbed operator $\widehat{H}_\hbar$. In this case, $o(\hbar\eps_\hbar)$-quasimodes behave like eigenmodes of the unperturbed Hamiltonian $\widehat{H}_\hbar$:
\begin{propop}\label{p:localization} 
 Suppose that $\omega$ is partially Diophantine, that $(\varepsilon_\hbar)$ is $o(\hbar^{1+\gamma(\omega)})$ and that $(r_\hbar)$ is a sequence of non-negative real numbers that is $o(\eps_\hbar\hbar)$. Then $\mathcal{Q}(\widehat{P}_\hbar, r_\hbar) \subseteq \mathcal{M}(\widehat{H}_\hbar)$; in particular, as  consequence of the result of \cite{ArMa20}, given any $\mu\in\mathcal{Q}(\widehat{P}_\hbar, r_\hbar)$, there exists $\cE\in\Sigma_\cH$ such that 
\[
\supp\mu \subseteq \cH^{-1}(\cE). 
\]
\end{propop}
This result follows from the fact for this size of perturbation, quasimodes of sufficiently small width are close to eigenfunctions of the unperturbed Hamiltonian.
\begin{lemma}
\label{l:projection_lemma}
Suppose that $\omega$ is partially Diophantine and that $\varepsilon_\hbar=o(h^{1+\gamma(\omega)})$. Let $(\widehat{L}_\hbar)$ be a uniformly bounded family of self-adjoint operators on $L^2(\IR^d)$ such that $[\widehat{H}_\hbar, \widehat{L}_\hbar]=0$.  If 
\begin{equation}
\label{e:quasimode_equation_star}
(\widehat{H}_\hbar+\eps_\hbar \widehat{L}_\hbar)\psi_\hbar= \lambda_\hbar \psi_\hbar+ o(\varepsilon_\hbar\hbar),\quad   \| \psi_\hbar \|_{L^2(\R^d)}=1,\quad \lim_{\hbar\to 0^+}\lambda_\hbar =  1,
\end{equation}
then there exists a sequence $(\Psi_\hbar)$ of normalized eigenvectors for $\widehat{H}_\hbar$ with eigenvalues $\Lambda_\hbar$ such that
\[
\psi_\hbar = \Psi_\hbar + o(1),\quad  \lambda_\hbar= \Lambda_\hbar+O(\eps_\hbar),
\]
and $(\Psi_\hbar)$ is a quasimode of width $o(\hbar)$ of $\widehat{L}_\hbar$ with quasieigenvalue $\eps_\hbar^{-1}(\lambda_\hbar-\Lambda_\hbar)$.
\end{lemma} 
\begin{proof}
By hypothesis, $(\psi_\hbar)$ is a quasimode of $\widehat{H}_\hbar$ of width $O(\varepsilon_\hbar)$:
\begin{equation}\label{e:hqm}
\| (\widehat{H}_\hbar - \lambda_\hbar) \psi_\hbar \|_{L^2(\R^d)} = O(\varepsilon_\hbar).
\end{equation}
Let $(\delta_\hbar)$ be a sequence of positive real numbers that is $o(\hbar^{1+\gamma{(\omega})})$ and such that $\lim_{\hbar\to 0^+}\eps_\hbar/\delta_\hbar=0$. Let $\Pi_\hbar$ denote the spectral projector of $\widehat{H}_\hbar$ on the interval  $I_\hbar:= [ \lambda_\hbar-\delta_\hbar,\lambda_\hbar+\delta_\hbar]$. Then, using \eqref{e:spacing}, we have
\begin{align*}
\|\psi_\hbar - \Pi_\hbar\psi_\hbar\|_{L^2(\R^d)} 
\leq \Vert (\widehat{H}_\hbar - \lambda_\hbar)^{-1}( \operatorname{Id} - \Pi_\hbar) \Vert_{\mathcal{L}(L^2)} \Vert (\widehat{H}_\hbar - \lambda_\hbar) \psi_\hbar \Vert_{L^2(\R^d)} = O(\varepsilon_\hbar\delta_\hbar^{-1}).
\end{align*}
The lower bound on the eigenvalue spacing \eqref{e:spacing} ensures that there exists a unique $\Lambda_\hbar\in\Spec(\widehat{H}_\hbar)\cap I_\hbar$, which must also satisfy $|\Lambda_\hbar-\lambda_\hbar|=O(\eps_\hbar)$, by \eqref{e:hqm}.

As a conclusion, $\Psi_\hbar:= \| \Pi_\hbar\psi_\hbar \|_{L^2(\R^d)}^{-1} \Pi_\hbar\psi_\hbar$ form a sequence of normalized eigenfunctions with eigenvalues $\Lambda_\hbar$ such that:
\[
\|\psi_\hbar - \Psi_\hbar\|_{L^2(\R^d)}=o(1).
\]
Applying $\Pi_\hbar$ to both sides of \eqref{e:quasimode_equation_star} and using that $[\Pi_\hbar,\widehat{L}_\hbar]=0$, we deduce:
\[
\Lambda_\hbar \Psi_\hbar +\varepsilon_\hbar \widehat{L}_\hbar\Psi_\hbar = \lambda_\hbar \Psi_\hbar +o(\varepsilon_\hbar\hbar).
\]
In other words:
\[
\widehat{L}_\hbar\Psi_\hbar= \frac{\lambda_\hbar - \Lambda_\hbar}{\varepsilon_\hbar}\Psi_\hbar+o(\hbar).
\]
\end{proof}

\begin{proof}[Proof of Proposition \ref{p:localization}]
Since, by Proposition \ref{averaging-lemma},
\[
\widehat{P}_{1,\hbar}:=U_{1,\hbar}^*\widehat{P}_\hbar U_{1,\hbar}=\widehat{H}_\hbar +  \langle \widehat{V}_{\hbar} \rangle + O_{\mathcal{L}(L^2)}(\varepsilon_\hbar^{2}),
\]
it follows that when $(\psi_\hbar)$ is a quasimode for $\widehat{P}_\hbar$ of width $r_\hbar=o(\varepsilon_\hbar\hbar)$ then $(U_{1,\hbar}^*\psi_\hbar)$ is a quasimode for $\widehat{P}_{1,\hbar}$ of width $U_{1,\hbar}^*r_\hbar=o(\varepsilon_\hbar\hbar)$. Lemma \ref{l:projection_lemma} ensures the existence of a sequence $(\Psi_\hbar)$ of normalized eigenvectors for $\widehat{H}_\hbar$ with eigenvalues $\Lambda_\hbar\to 1$ as $\hbar\to 0^+$ such that
$U_{1,\hbar}^*\psi_\hbar - \Psi_\hbar \to 0$ in $L^2(\R^d)$. Identity \eqref{FIO-estimate} implies that $(U_{1,\hbar}^*\psi_\hbar)$ and $(\psi_\hbar)$ have the same semiclassical measure, and the result follows from Theorem 1 in \cite{ArMa20}.
\end{proof}
We now have all the ingredients needed to prove Theorem \ref{t:quasimodes_2}.
\begin{proof}[Proof of Theorem \ref{t:quasimodes_2}]
Let $(\Psi_\hbar)$ be a sequence of normalized eigenfunctions for $\widehat{H}_\hbar$ with eigenvalues $\Lambda_\hbar\to 1$ as $\hbar\to 0^+$ and whose semiclassical measure $\mu$ is the Haar measure on a torus $\mathcal{T}_E$ for some $E\in\R^d_+$ with $|E|_1=1$. Given $N > 1 + \alpha$, we consider the unitary operator $U_{N,\hbar}$ given by Proposition \ref{averaging-lemma} to construct the quantum normal form of $\widehat{P}_\hbar$. Defining $\psi_\hbar:= U_{N,\hbar} \Psi_\hbar$ and using \eqref{FIO-estimate}, one has that the semiclassical measure associated to the sequence $(\psi_\hbar)$ is $\mu$. On the other hand, using the quantum normal form \eqref{quantum_normal_form}, we obtain the existence of a sequence  $(\lambda_\hbar)$ of quasi-eigenvalues tending to one as $\hbar\to 0 ^+$ such that
$$
\widehat{P}_\hbar \psi_\hbar = \lambda_\hbar \psi_\hbar + O(\varepsilon_\hbar^{N+1}).
$$
The quasi-eigenvalues $\lambda_\hbar$ are obtained as follows: the non-resonant condition on $\omega$ implies that the spectrum \eqref{e:spectrum} is simple. Therefore $\Psi_\hbar$ is the eigenfunction of $\widehat{H}_\hbar$ corresponding to the eigenvalue 
\[
\Lambda_\hbar=\hbar \omega\cdot k^\hbar+\frac{\hbar|\omega|_1}{2}
\]
for some $k^\hbar = (k_1^\hbar, \ldots, k_d^\hbar) \in \mathbb{N}^d$ and is also an eigenfunction of all the operators appearing in the quantum Birkhoff normal form \eqref{quantum_normal_form}. Then
$$
\lambda_\hbar = \Lambda_\hbar + \sum_{j=1}^N \varepsilon_\hbar^j  f_{j,\hbar}  \big(\hbar k_1^\hbar + \hbar/2, \ldots, \hbar k_d^\hbar + \hbar/2\big),
$$ 
for certain smooth functions $f_{j,\hbar}$.
This proves the first statement. The second statement is a direct consequence of of this together with Proposition \ref{p:localization}.
\end{proof}

\section{Propagation of coherent and excited states by the multiflow}\label{s:cs}

In this section we revisit and adapt the results on propagation of coherent and excited states of \cite{Robert97} (see also \cite[Chpts. 3 and 4]{Robert12}) to the case of a family of commuting quantum propagators. This will be used in Section \ref{s:qm} to construct quasimodes concentrating on orbits issued from any given point $z_0 \in H^{-1}(1)$ by the multiflow associated to that family of propagators. One heavily uses the fact that the propagation of a coherent or excited state by a time-dependent quadratic Hamiltonian, appearing as the Taylor approximation of degree two of the full Hamiltonian around the classical multiorbit, is determined solely in terms of the associated classic flow (governing the translation of the state) and of the linearized dynamical system via a metaplectic operator (governing its metric) together with a phase factor given by the action of the Hamiltonian flow. Moreover, one can approximate the evolution of the full quantum propagator using Duhamel's principle and estimating the coefficients of the solution in an orthonormal basis of excited states. 

Let $L\in S^0(T^*\R^d)$, denote by $\phi_s^{\langle L \rangle}$ the Hamiltonian flow associated to $\langle L \rangle$ and define its quantization $\widehat{L}_h := \Op_\hbar(L)$. Our goal is to give the necessary elements of the proof of Theorem \ref{t:propagator_theorem} which gives a semiclassical approximation of the evolution of a coherent state centered at $z_0 \in H^{-1}(1)$ by the quantum multiflow $e^{-it \widehat{H}_\hbar/\hbar }e^{-is \langle \widehat{L}_\hbar\rangle /\hbar}$.

\subsection{Coherent and excited states}
Let us consider the orthonormal basis of $L^2(\R^d)$ given by the family of normalized semiclassical Hermite functions
\begin{equation}
\label{e:hermite_function}
\Psi_0^\hbar(x) := \frac{1}{( \pi \hbar)^{d/4}} e^{-\frac{ \vert x \vert^2}{2\hbar}}; \quad \Psi_\nu^\hbar(x) = \frac{1}{\sqrt{ 2^{\vert \nu \vert} \nu!}} q_\nu\left( \frac{x}{\sqrt{\hbar}} \right) \Psi_0^\hbar(x), \quad \nu \in \mathbb{Z}^d_+,
\end{equation} 
where $q_\nu$ are Hermite polynomials defined by the three term relations:
$$
q_0(x) = 1, \quad (q_{\nu + e_j}(x))_{j=1}^d = x q_\nu(x) - e_j \cdot \operatorname{Id} \nabla q_\nu(x), \quad j= 1, \ldots, d.
$$
The set of Hermite functions $(\Psi_\nu^\hbar)$ gives an orthonormal basis of $L^2(\R^d)$ consisting of eigenfunctions for the Hamornic oscillator $\widehat{H}_\hbar$ with eigenvalues $\{\lambda_{\nu,\hbar} \, : \, \nu \in \mathbb{Z}^d_+  \}$ defined by \eqref{e:spectrum}. 

One can translate these states to a general point $z_0 = (x_0,\xi_0) \in \R^{2d}$ via the Weyl-Heisenberg translation operator $\widehat{T}(z_0)$ defined for any $\psi \in L^2(\R^d)$ by: 
$$
\widehat{T}(z_0) \psi(x) := \exp\left( - \frac{i}{2\hbar} x_0 \cdot \xi_0 \right) \exp \left( \frac{i}{\hbar} x \cdot \xi_0 \right) \psi(x-x_0), \quad \psi \in L^2(\R).
$$
The operator $\widehat{T}(z_0)$ is unitary and then the family of translated states 
$$
\Psi_{z_0,\nu}^\hbar := \widehat{T}(z_0)\Psi_{\nu}^\hbar, \quad \nu \in \mathbb{Z}_+^d,
$$
called coherent ($\nu = 0$) and excited ($\vert \nu \vert \geq 1$) states centered at $z_0$, is still an orthonormal basis of $L^2(\R^d)$. We denote for simplicity $\Psi_{z_0}^\hbar := \Psi_{z_0,0}^\hbar$. In the sequel, we study the propagation of the coherent state $\Psi_{z_0}^\hbar$ by the quantum flow $e^{-it \widehat{H}_\hbar/\hbar }e^{-is \langle \widehat{L}_\hbar\rangle /\hbar}$.

\subsection{Classical dynamical quantities}
\label{ss:classical_dynamical_quantities}
This section is devoted to introduce some classical quantities that will be used to describe the propagation along the minimal tori issued from a point $z_0 \in H^{-1}(1)$ by the flow $\phi_t^H$. It is convenient to introduce a new parametrization of the classical multiflow in order to simplify the analysis when dealing with those points satisfying $\mathbb{M}(z_0)\in\Sigma$. Let $E \in \R^{d}_+$ and put $d_E=d_{\pi_{E}(\omega)}$, where $\pi_E$ is given by \eqref{e:projection_degeneracies}.
Define $\ell_{E}:\{1,\hdots,d_E\}\To\{1,\hdots,d\}$ such that that the vectors $\pi_{E}(\nu_{\ell_{E}(1)}), \ldots, \pi_{E}(\nu_{\ell_{E}(d_E)})$ are linearly independent. 
We define $\widetilde{\cH}^{E} = ( \tilde{\mathcal{H}}_1^{E}, \ldots,  \tilde{\mathcal{H}}_{d_E}^{E})$ and $\tilde{v}_E\in\R^{d_E}$ by
\begin{equation}\label{e:hvtilde}
\tilde{\mathcal{H}}_j^{E}(z) := \frac{1}{K_{E,j}} \pi_{E}(\nu_{\ell_{E}(j)}) \cdot\mathbb{M}(z), \quad \widetilde{v}_{E,j} = K_{E,j} \left( v_{\ell_{E}(j)} + \sum_{n\in I_{E}} b_{n,j} v_n \right), \quad j = 1, \ldots, d_E,
\end{equation}
where $K_{E,j}$ is the greatest common divisor of the components of the vector $\pi_E(\nu_{\ell_E(j)})$ (which in general is not primitive), $I_{E}:=\ell_{E}(\{1,\hdots,d_E\})^c $, and the coefficients $b_{n,j} \in \mathbb{Q}$ satisfy
$$
\pi_{E}(\nu_n) = \sum_{j=1}^{d_E} b_{n,j} \pi_{E}(\nu_{\ell_{E}(j)}), \quad n \in I_{E}.
$$
Let us also define the flow
\begin{equation}
\label{e:minimal_flow}
\Phi_\tau^{\widetilde{\mathcal{H}}^{E}} := \phi_{\tau_1}^{\widetilde{\mathcal{H}}^{E}_1}\circ\hdots\circ\phi_{\tau_{d_E}}^{\widetilde{\mathcal{H}}^{E}_{d_E}},\quad \tau\in\R^{d_E}.
\end{equation}
Note that when $E$ satisfies that $d_E=d_\omega$ one has $\Phi_\tau^{\widetilde{\mathcal{H}}^{E}}=\Phi_\tau^\cH$ and, in general, 
\[
\Phi_{\tau}^{\widetilde{\mathcal{H}}^{E}}=\Phi_{\tilde{\rho}_E(\tau)},\quad \forall\tau\in\R^{d_E},
\] 
where $\tilde{\rho}_E:\R^{d_E}\To \R^{d}$ is the linear map sending the vectors of the canonical basis to the vectors $\pi_E(\nu_{\ell_E(j)})$ with $j=1,\hdots,d_E$.
When $\mathbb{M}(z_0)=E$,
\[
\Phi_{t\tilde{v}_E}^{\widetilde{\mathcal{H}}^{E}}(z_0)=\Phi_{tv}^\cH(z_0)=\phi^H_t(z_0), \quad t \in \R.
\]
Let us consider the partial orbit $\mathcal{C}_{z_0}(T)$  issued from $z_0$ by the flows $\Phi_t^{\widetilde{\cH}^E}$ and $\phi_s^{\langle L \rangle}$, that is:
$$
\mathcal{C}_{z_0}(T) :=  \{ \phi_s^{\langle L \rangle} \circ \Phi^{\widetilde{\cH}^{E}}_\tau(z_0)\, : \, (\tau,s) \in \mathbb{T}^{d_E} \times [0,T] \},
$$
where $E = \mathbb{M}(z_0)$ and $\Phi^{\widetilde{\cH}^{E}}_\tau$ is defined by \eqref{e:minimal_flow}. Since  $H_1,\hdots, H_d$ generate a completely integrable system, the partial orbit  $\mathcal{C}_{z_0}(T)$ is an embedded submanifold of $H^{-1}(1)$ (possibly with boundary) and there exists a diffeomorphism $\Theta_{z_0,T} : \mathcal{C}_{z_0}(T) \to \mathbb{T}^{d_E} \times \Gamma_T$, where $\Gamma_T$ is either $[0,T]$, $\mathbb{T}^1$ or $\{ 0 \}$ such that
$$
\Theta_{z_0,T} \circ (\phi_s^{\langle L \rangle} \circ \phi_t^H ) \circ \Theta_{z_0,T}^{-1}(\tau,r) = (\tau + t \tilde{v}_E +  s v_1 , r + \rho s),\quad v_1\in \R^{d_E},\; \rho\in\R,
$$
where $\tilde{v}_E$ is defined in \eqref{e:hvtilde} and provided that $r+\rho s \in \Gamma_T$. Notice that if $\Gamma_T = \{ 0 \}$ then one has $\rho = 0$.  

Define the symplectic matrix
\[
f(t,s) := d \big( \phi_t^H \circ \phi_s^{\langle L \rangle}(z) \big) \vert_{z = z_0},
\]
and denote
\begin{align}
\theta(z_0,T) := \sup_{0 \leq s,s' \leq T}   \sup_{t,t' \in \R} \big[ \tr \big(f(t,s)^\textnormal{t}f(t',s')\big) \big]^{1/2}.
\end{align}
\begin{remark}
\label{r:discussion_Ehrenfest}
The quantity $\theta(z_0,T)$ is related to the \textit{Ehrenfest time} of the quantum propagator $e^{-it \widehat{H}_\hbar/\hbar }e^{-is \langle \widehat{L}_\hbar\rangle /\hbar}$. Given $0 < \epsilon < \frac{1}{2}$, we define: 
$$
\mathscr{T}_\hbar := \sup_{T > 0} \big \{ T \, : \, \hbar^{1/2} \theta(z_0,T) \leq \hbar^\epsilon \big \}.
$$
\begin{itemize} 
\item If the orbit issued from $z_0$ is stable, say $\theta(z_0,T) \leq C(1+ \vert T \vert)$, then:
\begin{align}
\label{e:ehrenfest_stable}
\mathscr{T}_\hbar = C \hbar^{\epsilon - \frac{1}{2}}.
\end{align}

\item If the orbit issued from $z_0$ is unstable, one has $\theta(z_0,T) \leq C e^{\lambda(z_0) T}$ for some $\lambda(z_0)>0$. Then
\begin{align}
\label{e:ehrenfest_unstable}
 \mathscr{T}_\hbar = \frac{1 - 2 \epsilon}{2\lambda(z_0)}  \log \left(\frac{1}{\hbar} \right).
\end{align}
\end{itemize}
\end{remark}
We also define $F_{z_0}(\tau,s) := d\big( \Phi_\tau^{\widetilde{H}^E} \circ \phi_s^{\langle L \rangle}(z) \big)\vert_{z = z_0} \in \operatorname{Sp}(2d,\R)$ for $(\tau,s) \in \mathbb{T}^{d_E} \times \R$. Clearly $f(t,s) = F_{z_0}(t \widetilde{v}_E, s)$. Moreover, $F_{z_0}(\tau,s)$ is  the solution of the following system of differential equations:
\begin{align*}
\frac{d}{ds} F_{z_0}(\tau,s) & = J \partial^2 \langle L \rangle F_{z_0}(\tau,s), \quad  \partial^2 \langle L \rangle = \left( \frac{ \partial^2 \langle L \rangle}{\partial z^2} \right) \Big \vert_{z = \phi_s^{\langle L \rangle} \circ \Phi^{\widetilde{\cH}^E}_\tau(z_0)}, \\[0.2cm]
\frac{d}{d\tau_j} F_{z_0}(\tau,s) & = J \partial^2\widetilde{\cH}^E_j F_{z_0}(\tau,s), \quad \, \partial^2 \widetilde{\cH}^E_j = \left( \frac{ \partial^2  \widetilde{\cH}^E_j }{\partial z^2} \right) \Big \vert_{z =  \phi_s^{\langle V \rangle} \circ \Phi^{\tilde{\cH}^E}_\tau(z_0)}, 
\end{align*} 
with $F_{z_0}(0,0) = \operatorname{Id}$ and $J$ being the standard symplectic matrix
$$
J = \left( \begin{array}{cc}
 0 & \operatorname{Id} \\[0.2cm]
 - \operatorname{Id} & 0 
 \end{array} \right).
$$
In particular, since $\{ \widetilde{\cH}_j^E, \langle L \rangle \} = 0$,
\begin{align*}
X_{\langle L \rangle}(\phi_s^{\langle L \rangle} \circ \Phi_\tau^{\widetilde{\cH}^E}(z_0)\big) & = F_{z_0}(\tau,s)X_{\langle L \rangle}(z_0), \\[0.2cm]
X_{\widetilde{\cH}^E_j}\big(\phi_s^{\langle L \rangle} \circ \Phi^{\widetilde{\cH}^E}_\tau(z_0)\big) & = F_{z_0}(\tau,s)X_{\widetilde{\cH}_j^E}(z_0).
\end{align*} 

As we will see in the next section, the symplectic matrix $F_{z_0}(\tau ,s)$ associated to the linearized problem governs the change of the metric of the coherent state $\Psi_{z_0}^\hbar$ when it is propagated by the quantum flow, via the action of a metaplectic operator $\widehat{F}_{z_0}(\tau,s)$ quantizing $F_{z_0}(\tau , s)$.

\subsection{Quantum propagation}

The symplectic matrix $F_{z_0}(\tau,s)$ can always be written as
$$
F_{z_0}(\tau,s) =  \left( \begin{array}{cc} \operatorname{Re} Q(\tau,s) & \operatorname{Im} Q(\tau,s) \\[0.2cm]
 \operatorname{Re} P(\tau,s) & \operatorname{Im} P(\tau,s) \end{array} \right),
 $$
for some complex $d\times d$ matrices $Q(\tau,s)$, $P(\tau,s)$ such that $B(\tau,s) = P(\tau,s)Q(\tau,s)^{-1}$ belongs to Siegel upper half-space. There exists a metaplectic operator $\widehat{F}_{z_0}(\tau,s)$ that acts on each $\Psi_\nu^\hbar$ (see \cite[(2.22)]{Robert97} and \cite[Remark 3.2, and Thm. 3.4]{Hagedorn98}), as
\begin{align}
\label{e:hermite_1}
\widehat{F}_{z_0}(\tau,s) \Psi_\nu^\hbar  = \frac{1}{\sqrt{2^{\vert \nu \vert} \nu!}} \frac{\det Q(\tau,s)^{-1/2}}{(\pi \hbar)^{d/4}} p^M_\nu \left( \frac{ Q(\tau,s)^{-1} x}{\sqrt{\hbar}}  \right)  \exp \left( \frac{i}{2\hbar} B(\tau,s) x \cdot x \right), \quad \nu \in \mathbb{Z}^d_+,
\end{align}
where $p^M_\nu$ are Hermite-type polynomials satisfying $p^M_\nu = q_\nu$ at the initial value $(\tau,s) = (0,0)$, where $M= \operatorname{Id}$, for the standard Hermite polyonomials $q_\nu$ given in \eqref{e:hermite_function}, and with coefficients defined by universal three-term recursion relations in terms of $F_{z_0}(\tau,s)$ (see \cite{Hagedorn85}, \cite{Hagedorn98}, and more precisely \cite[Prop. 4]{Dietert17}):
\begin{equation}
\label{e:general_hermite_polynomials}
p_0^M(x) = 1, \quad (p_{\nu + e_j}^M(x))_{j=1}^d = x p_\nu^M(x) - e_j \cdot M\nabla p_\nu^M(x), \quad j= 1, \ldots, d,
\end{equation}
where $M = M(\tau,s) = Q(\tau,s)^{-1} \overline{Q (\tau,s)}$. This operator plays a central role in the expression of the quantum propagation of coherent and excited states.

Consider now the quantum propagator associated to $\widetilde{\cH}^E$, where $\mathbb{M}(z_0)=E$:
\begin{equation}
\label{e:unitary_operator}
\mathbf{U}^E_\hbar(\tau,s) := 
 \exp \left( -\frac{i}{\hbar}  \left(\sum_{j=1}^{d_E}\tau_j   \Op_\hbar( \tilde{\mathcal{H}}^E_j)+s \langle \widehat{L}_\hbar \rangle  \right)\right).
\end{equation}
Applying \cite[Thm 3.1]{Robert97} successively for the propagators of the commuting operators $\Op_\hbar(\tilde{\mathcal{H}}^E)$ and  $\langle \widehat{L}_\hbar \rangle$, we get the following approximation of $\mathbf{U}_\hbar(\tau,s) \Psi_{z_0}^\hbar(x)$ in terms of a linear combination of excited coherent states:
\begin{teor}
\label{t:propagator_theorem}
Let $z_0\in H^{-1}(1)$ and $T>0$; write $E:=\mathbb{M}(z_0)$ and define
\begin{align*}
\psi_\nu^\hbar(\tau,s,x) &  := e^{\frac{i}{\hbar}(\gamma(\tau,s)-\sum_{j=1}^{d_E}\tau_j \tilde{\cH}^E_j(z_0))} \, \widehat{T}\big( \phi_s^{\langle L \rangle} \circ \Phi_\tau^{\widetilde{\cH}^E}(z_0) \big) \,\widehat{F}_{z_0}(\tau,s) \Psi_\nu^\hbar(x),
\end{align*}
where the Hermite functions $\Psi_\nu^\hbar$ are given by \eqref{e:hermite_function} and $\gamma(\tau,s)$ is the classical action 
\begin{equation}
\label{e:action_gamma}
\gamma(\tau,s) = -s\langle L \rangle(z_0) + \int_0^s \frac{\partial_r x(\tau,r) \cdot \xi(\tau,r) - x(\tau,r) \cdot \partial_r\xi(\tau,r)}{2} dr,
\end{equation}
where $z(\tau,s) := \big( x(\tau,s), \xi(\tau,s) \big) = \phi_s^{\langle L \rangle} \circ \Phi^{\tilde{\cH}^E}_\tau(z_0)$.

Then for every $l \geq 1$, 
\begin{equation}
\label{e:linear_combination_of_coherent_states}
\mathbf{U}^E_\hbar(\tau,s) \Psi_{z_0}^\hbar(x) = \sum_{\vert \nu \vert = 0}^{3(l-1)} c_\nu(\tau,s,\hbar) \psi_\nu^\hbar(\tau,s,x) + \mathcal{R}_l(x,z_0,\hbar,\tau,s),
\end{equation}
where the wave-front set of the remainder satisfies $\operatorname{WF}_\hbar \big( \mathcal{R}_l(\cdot,z_0,\tau,s,\hbar) \big) = \{ z(\tau,s) \} $ and
\begin{equation}
\label{e:reminder_estimate_ehrenfest}
\sup_{(\tau,s) \in  \mathbb{T}^{d_E} \times [0,T]} \Vert \mathcal{R}_l(\cdot,z_0,\hbar, \tau,s) \Vert_{L^2(\R^d)} \leq C_l \sum_{1 \leq j \leq l} \left( \frac{  T }{\hbar} \right)^j \big( \sqrt{\hbar} \theta(z_0,T) \big)^{2j + l}.
\end{equation}
The constants $c_\nu$  (see\footnote{Actually, estimate \cite[Eq. (3.7)]{Robert97} can be slightly precised to obtain:
$$
\left \vert c_\nu(\tau,s,\hbar) - c_\nu(0,0) \right \vert \leq \sum_{1 \leq p \leq l-1} \sum_{\substack{\vert \mu - \nu \vert \leq k_1 + \cdots + k_p \leq l + 2p - 1 \\ k_i \geq 3}} \vert a_{p,\mu,\nu}(\tau,s) \vert \vert c_{\mu}(0,0) \vert \hbar^{\frac{k_1+ \cdots + k_p}{2}-p},
$$
where the entries $a_{p,\mu,\nu}(\tau,s)$ are universal polynomials in $((\widetilde{\cH}^E_j)^{ (\gamma)}(z(\tau,s)), \langle L \rangle^{(\gamma')}(z(\tau,s)))$ for $\vert \gamma + \gamma' \vert \leq l+2$ and $a_{p,\mu,\nu}(0,0) = 0$.
In our case, $c_0(0,0) = 1$ and $c_\nu(0,0) = 0 $ for $\nu \neq 0$.} \cite[Eq. (3.7)]{Robert97}) satisfy $\vert c_0(\tau,s,\hbar) - 1 \vert \leq C_l \hbar^{1/2}$, and, for every $1 \leq \vert \nu \vert \leq 3(l-1)$ 
\begin{align}
\label{e:constant_estimate}
\sup_{(\tau, s) \in \mathbb{T}^{d_E} \times [0, T]} \vert c_\nu(\tau,s,\hbar) \vert & \leq  C_l \hbar^{\frac{\vert \nu \vert}{2}},
\end{align}
where $C_l$ does not depend on $T$.

\end{teor}

\section{Construction of quasimodes}\label{s:qm}

In this section we present a construction of quasimodes of width $O(\varepsilon_\hbar \hbar)$ for the operator 
\[
\widehat{H}_\hbar + \delta_\hbar \langle \widehat{L}_\hbar \rangle,\quad \delta_\hbar\to 0,\quad \widehat{L}_\hbar = \Op_\hbar(L), \quad L\in S^0(T^*\R^d).
\]
We will use propagation of coherent-states and the fact that $\widehat{H}_\hbar$ and $\widehat{H}_\hbar + \delta_\hbar \langle \widehat{L}_\hbar \rangle$ commute to construct, for any given $z_0 \in H^{-1}(1)$, a quasimode with semiclassical measure given by the average of $\mu_{\omega, z_0}$ by the flow $\phi_s^{\langle L \rangle}$ in a finite interval $s \in [0,T]$. 
\begin{lemma}[Concentration on a minimal set]
\label{concentration_minimal_set_2}
For every $\cE\in\Sigma_\cH$
there exits $\hbar_0>0$ and  a sequence $\lambda_\hbar \to 1$ such that for every $z_0 \in \cH^{-1}(\cE)$, $\chi \in \mathcal{C}_c^\infty((0,1),[0,1])$ there exists a family $(\psi_\hbar^T) \subset L^2(\R^d)$, parametrized by $T \in (0,\mathscr{T}_\hbar)$, of quasimodes for $\widehat{H}_\hbar + \delta_\hbar \langle \widehat{L}_\hbar \rangle$ with quasi-eigenvalue $\lambda_\hbar$ and width $r_\hbar = c(T,\hbar) \delta_\hbar \hbar $, where
\begin{align}
\label{e:constant_c}
c(T,\hbar) =  \frac{ C_\chi}{T} \left( 1 +  O\big( \hbar^{1/2}\theta(z_0,T) \big) \right), \quad C_\chi>0,
\end{align}
such that, for every $a \in \mathcal{C}_c^\infty(\R^{2d})$, $0 < \hbar \leq \hbar_0$, and $0< T < \mathscr{T}_\hbar$,
\begin{equation}
\label{e:best_concentration_2}
 \big \langle \psi_\hbar^T, \Op_\hbar(a) \psi_\hbar^T \big \rangle_{L^2(\R^d)} =  \frac{1}{ \Vert \chi_T \Vert_{L^2(\R)}^2} \int_\R \chi_T(s)^2 \, \langle a \rangle \circ \phi_s^{\langle L \rangle}(z_0) ds  +  O \big(\hbar^{1/2}\theta(z_0,T) \big),
\end{equation}
where $\chi_T:=\chi(\cdot/T)$. In particular,
$$
\mu_T = \frac{1}{ \Vert \chi_T \Vert_{L^2(\R)}^2} \int_\R \chi_T(s)^2 \big( \phi_s^{\langle V \rangle} \big)_* \mu_{\omega,z_0}ds  \in \mathcal{Q}(\widehat{H}_\hbar + \delta_\hbar \langle \widehat{L}_\hbar \rangle, r_\hbar).
$$
\end{lemma}

\begin{remark}
Note that \eqref{e:best_concentration_2} implies in particular that for $T>0$ the quasimodes are asymptotically normalized:
\[
\lim_{\hbar\to 0^+}\|\psi_\hbar^T\|_{L^2(\R^d)}=1.
\]
\end{remark}

\begin{proof}

Let $E:=\mathbb{M}(z_0)$. We first assume that $\Gamma_T =[0,T]$ or $\Gamma_T = \mathbb{T}^1$, and consider the map
\begin{equation}
\label{e:parametrization}
\mathbb{T}^{d_E} \times [0,T] \ni (\tau,s) \mapsto z(\tau,s) := \phi_s^{\langle L \rangle} \circ \Phi_\tau^{\widetilde{\cH}^{E}}(z_0)\in \mathcal{C}_{z_0}(T).
\end{equation}
Let $M_\hbar$ be a sequence given by
\begin{equation}
\label{e:multieigenvalue}
M_\hbar :=  \big( \widetilde{\lambda}_\hbar^1, \ldots, \widetilde{\lambda}_\hbar^{d_E} \big) = (\widetilde{\cH}^E_1(z_0),\hdots, \widetilde{\cH}^E_{d_E}(z_0)) + O(\hbar),
\end{equation}
where $\widetilde{\lambda}_\hbar^j$ for $j = 1, \ldots, d_E$ are eigenvalues of $\Op_\hbar(\tilde{\cH}^E_j)$ that can be defined as in the proof of \cite[Lemma 1]{ArMa20}.
We define our candidate $(\psi_\hbar)$ to desired quasimode by:
\begin{equation}
\label{e:quasimode_I}
\psi_\hbar(x) := \left(\frac{C_T(z_0)}{ \sqrt{\hbar^{d_E+1}}}
\right)^{1/2} \;  \sum_{\vert \nu \vert = 0}^{3(l-1)} \int_{\mathbb{T}^{d_E} \times \R} \chi_T(s) e^{\frac{i(M_\hbar,E')\cdot (\tau,s)}{\hbar}} c_\nu(\tau,s,\hbar) \psi_\nu^\hbar(\tau,s,x)  \frac{d\tau}{|\T^{d_E}|}ds, 
\end{equation}
where $E'=\langle L \rangle(z_0)$, the constant $C_T(z_0) \in \R$ is given by
\begin{equation}
\label{e:constant_T}
C_T(z_0) = \sqrt{ \frac{\det \mathcal{G}_{z_0}}{\pi^{d_E +1} }} \cdot \frac{ \vert \mathbb{T}^{d_E}   \vert}{T \Vert \chi \Vert_{L^2}^2 }, \quad \mathcal{G}_{z_0} :=  \big( d \widetilde{\mathcal{H}}^E(z_0), d \langle L \rangle(z_0) \big) \, \big( d \widetilde{\mathcal{H}}^E(z_0), d \langle L \rangle(z_0) \big)^{\textnormal{t}},
\end{equation}
with $\vert \mathbb{T}^{d_E} \vert$ being the volume of $\mathbb{T}^{d_E}$.

We will show that this sequence defines a quasimode that satisfies all the claims of the theorem. We star by showing that \eqref{e:best_concentration_2} takes place. We first compute the cross-Wigner distribution of the functions: \begin{equation}
\varphi_\nu^\hbar(x) := \left(\frac{C_T(z_0)}{\sqrt{\hbar^{d_E+1}}} \right)^{1/2} \int_{\mathbb{T}^{d_E} \times \R} \chi_T(s) c_\nu(\tau,s,\hbar) \psi_\nu^\hbar(\tau,s,x) \frac{d\tau}{|\T^{d_E}|}ds, \quad \nu \in \mathbb{Z}^d_+.
\end{equation}
Introducing the notation $\mathbf{t} = (\tau,s)$ to shorten the formulas below, we lead to analyse the following oscillatory integral which, as we will see, has stationary-phase near the diagonal $\mathbf{t} \sim \mathbf{t}'$:
\begin{align*}
W_\hbar[\varphi_\nu^\hbar,\varphi_{\nu'}^\hbar] & \\[0.2cm] 
 & \hspace*{-1.5cm} = \frac{C_T(z_0)}{\sqrt{\hbar^{d_E+1}}} \int_{(\mathbb{T}^{d_E} \times \R)^2} \kappa_{\nu,\nu'}^T(\mathbf{t},\mathbf{t}',\hbar)e^{i\frac{ \gamma(\mathbf{t}) - \gamma(\mathbf{t}')}{\hbar}} W_\hbar \big[ \widehat{T}(z(\mathbf{t})) \widehat{F}_{z_0}(\mathbf{t}) \Psi_\nu^\hbar,  \widehat{T}(z(\mathbf{t}')) \widehat{F}_{z_0}(\mathbf{t}') \Psi_{\nu'}^\hbar \big] d \mathbf{t} d \mathbf{t}',
\end{align*}
where $\kappa^T_{\nu,\nu'}(\mathbf{t},\mathbf{t}',\hbar) = c_\nu(\mathbf{t},\hbar) \overline{c_{\nu'}(\mathbf{t}',\hbar)} \chi_T(s) \chi_T(s')$, the cross-Wigner function $W_\hbar[\varphi,\psi](x,\xi)$ is defined by
\begin{equation}
W_\hbar[\varphi,\psi](x,\xi) := \frac{1}{(2\pi)^d} \int_{\R^d} e^{\xi \cdot v} \varphi\left(x - \frac{\hbar v}{2} \right) \overline{ \psi\left( x + \frac{\hbar v}{2} \right)} dv,
\end{equation}
and  we have written $d \mathbf{t}  =  \frac{d\tau}{|\T^{d_E}|}ds$ for simplicty. An explicit computation (see for instance \cite[Eq. (9.25)]{DeGosson11}) gives us
\begin{align*}
W_\hbar\big[ \widehat{T}(z(\mathbf{t})) \widehat{F}_{z_0}(\mathbf{t}) \Psi_\nu^\hbar,  \widehat{T}(z(\mathbf{t}')) \widehat{F}_{z_0}(\mathbf{t}') \Psi_{\nu'}^\hbar \big](z) & \\[0.2cm]
 & \hspace*{-6.5cm} = \exp \left( -\frac{i}{2\hbar} \sigma( z(\mathbf{t}), z(\mathbf{t}')) -\frac{i}{\hbar} \sigma(z, z(\mathbf{t}) - z(\mathbf{t}')) \right) W_\hbar \big[ \widehat{F}_{z_0}(\mathbf{t}) \Psi_\nu^\hbar,  \widehat{F}_{z_0}(\mathbf{t}') \Psi_{\nu'}^\hbar   \big](z- \mathbf{z}(\mathbf{t},\mathbf{t}')),
\end{align*}
where $\mathbf{z}(\mathbf{t},\mathbf{t}') = \frac{1}{2} (z(\mathbf{t}) + z(\mathbf{t}'))$.
Moreover (see \cite[Prop. 244]{DeGosson11}, \cite[Thm. 6.1]{Troppmann2017}, \cite[Lemma 10]{Dietert17}, the structure of generalized Hermite functions is conserved by the Wigner transform, in the sense that
\begin{align*}
W_\hbar \big[ \widehat{F}_{z_0}(\mathbf{t}) \Psi_\nu^\hbar,  \widehat{F}_{z_0}(\mathbf{t}') \Psi_{\nu'}^\hbar   \big](z)
 & = \frac{1}{\sqrt{2^{\vert \nu \vert + \vert \nu' \vert} \nu! \nu'!}} \frac{\det \mathcal{Q}(\mathbf{t}, \mathbf{t}')^{-1/2}}{(\pi \hbar)^d} P^{\mathcal{M}}_{\nu,\nu'}\left( \frac{\mathcal{Q}(\mathbf{t}, \mathbf{t}')^{-1} z}{\sqrt{\hbar}} \right)  e^{\frac{i}{\hbar} \mathcal{B}(\mathbf{t},\mathbf{t}')z \cdot z} \\[0.2cm]
 & =: \mathfrak{H}_{\nu,\nu'}[\mathbf{t}, \mathbf{t}']\left( \frac{z}{\sqrt{\hbar}} \right),
\end{align*}
where
$$
\mathcal{Q}(\mathbf{t}, \mathbf{t}') = \frac{1}{2} \left( \begin{array}{cc}
\overline{P}(\mathbf{t}) & P(\mathbf{t}') \\[0.2cm]
\overline{Q}(\mathbf{t}) & Q(\mathbf{t}') 
\end{array} \right), \quad \mathcal{P}(\mathbf{t}, \mathbf{t}') = \left( \begin{array}{cc}
- J \overline{P}(\mathbf{t}) & J P(\mathbf{t}') \\[0.2cm]
- J \overline{Q}(\mathbf{t}) & J Q(\mathbf{t}')
\end{array} \right),
$$
$\mathcal{M} = \mathcal{Q}^{-1} \overline{\mathcal{Q}}$, the matrix $\mathcal{B} = \mathcal{P} \mathcal{Q}^{-1}$ lies in the Siegel upper half-space, and the polynomials $P^{\mathcal{M}}_{\nu,\nu'}$ are defined by the three-term recursion relation
$$
P_{0,0}(z) = 1, \quad (P_{(\nu,\nu') + e_j}^{\mathcal{M}}z)_{j=1}^{2d} = z P_{\nu,\nu'}^{\mathcal{M}}(z) - e_j \cdot \mathcal{M}\nabla P_{\nu,\nu'}(z), \quad j= 1, \ldots, 2d.
$$
In particular, when $\mathbf{t} = \mathbf{t}'$, 
$$
\mathcal{M}(\mathbf{t}, \mathbf{t}) = \mathcal{Q}^{-1} \overline{\mathcal{Q}} = \left( \begin{array}{cc}
0 & \operatorname{Id}_d \\[0.2cm]
\operatorname{Id}_d & 0 
\end{array} \right).
$$
In addition, we will need the following estimate that is proved (in a slightly less general situation) in \cite[Lemma 2.4]{Ar20}) on the Hermite functions $\mathfrak{H}_{\nu,\nu'}$: for every real  $2d \times (d_E +1)$ matrix $\mathcal{K}$ of rank $d_E+1$,
\begin{align}
 \label{e:estimate_norm__fourier_hermite}
 \sup_{0 \leq s,s' \leq T} \int_{\R^{d_E+1}} \vert \partial^\gamma \widehat{\mathfrak{H}}_{\nu,\nu'}[(\tau,s),(\tau',s')] ( \mathcal{K} \mathbf{r}) \vert d \mathbf{r} & \leq C_{\mathcal{K},\gamma} \, \theta(z_0,T)^{\vert \nu \vert + \vert \nu' \vert}, \quad \gamma \in \mathbb{N}^{2d}. 
\end{align}

On the other hand, the Wigner distribution $W_\hbar[\varphi_\nu^\hbar,\varphi_{\nu'}^\hbar]$ satisfies, when it acts on a test function $a \in \mathcal{C}_c^\infty(\R^d)$:
\begin{align*}
 W_\hbar[\varphi_\nu^\hbar,\varphi_{\nu'}^\hbar](a) & \\[0.2cm]
  & \hspace*{-2.2cm} = \frac{C_T(z_0)}{\sqrt{\hbar^{d_E+1}}} \int_{(\mathbb{T}^{d_E} \times \R)^2} \int_{\R^{2d}} \kappa_{\nu,\nu'}^T(\mathbf{t}, \mathbf{t}') a\big( \sqrt{\hbar} z + \mathbf{z}(\mathbf{t},\mathbf{t}') \big) e^{\frac{i}{\hbar} f_1(\mathbf{t},\mathbf{t}')} e^{\frac{i}{\sqrt{\hbar}} f_2(z,\mathbf{t},\mathbf{t}')} \mathfrak{H}_{\nu,\nu'}[\mathbf{t}, \mathbf{t}'](z) dz d \mathbf{t} d \mathbf{t}',
\end{align*}
where, denoting $\mathbf{E}_\hbar = (M_\hbar,E')$, the phase functions $f_1$ and $f_2$ are given by
\begin{align*}
f_1(\mathbf{t}, \mathbf{t}') & = \mathbf{E}_\hbar \cdot (\mathbf{t} - \mathbf{t}') + \frac{1}{2}\sigma(z(\mathbf{t}),z(\mathbf{t}')) + \gamma(\mathbf{t}) - \gamma(\mathbf{t}'), \\[0.2cm]
f_2(z,\mathbf{t},\mathbf{t}') & = x \cdot (\xi(\mathbf{t}) - \xi(\mathbf{t}')) - \xi \cdot (x(\mathbf{t}) - x(\mathbf{t}')).
\end{align*}
We next expand by Taylor in $\mathbf{t}'$ up to second order the functions $f_1$ and $f_2$ to localize the oscillatory integral near the diagonal. We first write
\begin{align*}
z(\mathbf{t}') & = z(\mathbf{t}) + (\mathbf{t}'- \mathbf{t}) \cdot \nabla_{\mathbf{t}} z(\mathbf{t}) + O(\vert \mathbf{t} - \mathbf{t}' \vert^2) \\[0.2cm]
\gamma(\mathbf{t}') & = \gamma(\mathbf{t}) + (\mathbf{t}'- \mathbf{t}) \cdot \nabla_{\mathbf{t}} \gamma(\mathbf{t}) + O(\vert \mathbf{t} - \mathbf{t'} \vert^2).
\end{align*}
Notice also that the derivative of the action $\gamma(\mathbf{t})$ satisfies
$$
\nabla_{\mathbf{t}} \gamma(\mathbf{t}) = \left( 0, -E + \frac{1}{2} \big( \partial_s x(\tau,s) \cdot \xi(\tau,s) - x(\tau,s) \cdot \partial_s \xi(\tau,s) \right),
$$
and the second term compensates with the linearization of the symplectic product $\sigma(z(\mathbf{t}), z(\mathbf{t}'))$:
$$
\sigma( z(\mathbf{t}), (\mathbf{t}' - \mathbf{t}) \cdot \nabla_{\mathbf{t}} z(\mathbf{t}) ) = (\tau'-\tau) \cdot M_\hbar + (s'-s)\frac{1}{2} \big( \partial_s x(\tau,s) \cdot \xi(\tau,s) - x(\tau,s) \cdot \partial_s \xi(\tau,s)\big).
$$
Thus, using \eqref{e:multieigenvalue}, we get
\begin{align}
f_1(\mathbf{t}, \mathbf{t}') & = (M_\hbar - \widetilde{\mathcal{H}}^E(z_0)) \cdot (\tau - \tau') + O(\vert \mathbf{t} - \mathbf{t}' \vert^3) = O(\hbar \vert \mathbf{t}- \mathbf{t}' \vert ) + O(\vert \mathbf{t} - \mathbf{t}' \vert^3) , \\[0.2cm]
\label{l:f_2}
f_2(z,\mathbf{t},\mathbf{t}') & = -\sigma \big( z,(\mathbf{t} - \mathbf{t}') \cdot \nabla_{\mathbf{t}} z(\mathbf{t}) \big) + z \cdot R(\mathbf{t}, \mathbf{t}'),
\end{align}
where $R(\mathbf{t}, \mathbf{t}') = O(\vert \mathbf{t} - \mathbf{t'} \vert^2)$.  Then,  changing coordinates $\mathbf{r} = \mathbf{t}' - \mathbf{t} \in \mathbb{T}^{d_E}\times\R$, we obtain:
\begin{align*}
\int_{\mathbb{T}^{d_E} \times \R} \kappa_{\nu,\nu'}^T(\mathbf{t}, \mathbf{t}') e^{\frac{i}{\hbar} f_1(\mathbf{t},\mathbf{t}')} \int_{\R^{2d}} e^{\frac{i}{\sqrt{\hbar}}  f_2(z,\mathbf{t},\mathbf{t}')}  a\big( \sqrt{\hbar} z + \mathbf{z}(\mathbf{t},\mathbf{t}') \big)  \mathfrak{H}_{\nu,\nu'}[\mathbf{t}, \mathbf{t}'] dz  d \mathbf{t}' \\[0.2cm]
 & \hspace*{-12.4cm} = \int_{ \mathbb{T}^{d_E}\times\R} \kappa_{\nu,\nu'}^T(\mathbf{t}, \mathbf{t}+ \mathbf{r}) e^{\frac{i}{\hbar} f_1(\mathbf{t},\mathbf{t} + \mathbf{r})} \\[0.2cm] 
 & \hspace*{-10cm} \times \int_{\R^{2d}}e^{\frac{i}{\sqrt{\hbar}} f_2(z,\mathbf{t},\mathbf{t} + \mathbf{r})}  a\big( z(\mathbf{t}) + \sqrt{\hbar} z +O_{\mathbf{t}}( \vert \mathbf{r} \vert)\big) \mathfrak{H}_{\nu,\nu'}[\mathbf{t}, \mathbf{t} + \mathbf{r}](z) dz \frac{d\mathbf{r}}{ \vert \mathbb{T}^{d_E}  \vert}
 \\[0.2cm]
 & \hspace*{-12.4cm} = \sqrt{\hbar^{d_E+1}} \int_{\frac{\Omega_{d_E}}{\sqrt{\hbar}}\times \R}   \kappa_{\nu,\nu'}^T(\mathbf{t}, \mathbf{t}+ \sqrt{\hbar}\mathbf{r}) e^{\frac{i}{\hbar} f_1(\mathbf{t},\mathbf{t} + \sqrt{\hbar}\mathbf{r})} \times \\[0.2cm]
  & \hspace*{-10cm} \times  \int_{\R^{2d}} e^{i f_2(z,\mathbf{t},\mathbf{t} + \sqrt{\hbar} \mathbf{r}) } a\big( z(\mathbf{t}) + \sqrt{\hbar} z +O_{\mathbf{t}}( \sqrt{\hbar} \vert \mathbf{r} \vert)\big) \mathfrak{H}_{\nu,\nu'}[\mathbf{t}, \mathbf{t} + \sqrt{\hbar} \mathbf{r}](z) dz \frac{d\mathbf{r}}{ \vert \mathbb{T}^{d_E}  \vert},
\end{align*}
where $\Omega_{d_E}\subseteq\R^{d_E}$ is a fundamental domain for $\T^{d_E}$.

On the other hand, by \eqref{l:f_2}, we can also write $-\sigma(z, \mathbf{r} \cdot \nabla_{\mathbf{t}} z(\mathbf{t})) = z \cdot \vartheta(\mathbf{r},\mathbf{t})$ for the vector
\begin{equation}
\label{e:vector}
\vartheta(\mathbf{r},\mathbf{t}) := \big( d\widetilde{\cH}^E(z(\mathbf{t})), d \langle L \rangle(z(\mathbf{t})) \big)^{\textnormal{t}} \mathbf{r},
\end{equation}
and observe that $\rk \big( d\widetilde{\cH}^E(z(\mathbf{t})), d \langle L \rangle(z(\mathbf{t})) \big) = d_E +1 $ by the hypothesis $\Gamma_T = [0,T]$ or $\Gamma_T = \mathbb{T}^1$. This matrix will play the role of $\mathcal{K}$ later on when using \eqref{e:estimate_norm__fourier_hermite}.
We then notice that the integral in $z$ performs the Fourier transform of the Hermite function  $\mathfrak{H}_{\nu,\nu'}[\mathbf{t}, \mathbf{t} + \sqrt{\hbar} \mathbf{r}](z)$ multiplied by the  function $a\big( z(\mathbf{t}) + \sqrt{\hbar} z +O_{\mathbf{t}}( \sqrt{\hbar} \vert \mathbf{r} \vert) \big)$, which is asymptotically constant as $\hbar \to 0$. By Taylor expansion of the test function $a$ around $z(\mathbf{t})$, we have
\begin{align*}
a\big( z(\mathbf{t}) + \sqrt{\hbar} z +O_{\mathbf{t}}( \sqrt{\hbar} \vert \mathbf{r} \vert) \big) & \\[0.2cm]
 & \hspace*{-3cm} = a(z(\mathbf{t})) + \big( \sqrt{\hbar} z + O_{\mathbf{t}}(\sqrt{\hbar} \vert \mathbf{r} \vert) \big) \cdot \int_0^1 \nabla a( z(\mathbf{t}) + s (\sqrt{\hbar} z +O_{\mathbf{t}}( \sqrt{\hbar} \vert \mathbf{r} \vert))) ds \\[0.2cm]
 & \hspace*{-3cm} =: a(z(\mathbf{t})) + \big( \sqrt{\hbar} z + O_{\mathbf{t}}(\sqrt{\hbar} \vert \mathbf{r} \vert) \big) \cdot r_a(\sqrt{\hbar} z,\mathbf{t},\sqrt{\hbar}\mathbf{r}).
\end{align*}
Taking the Fourier transform $\mathcal{F}_z$ in $z$, we observe that
\begin{align*}
\mathcal{F}_z \Big[ z \mathfrak{H}_{\nu,\nu'}[\mathbf{t}, \mathbf{t} + \sqrt{\hbar} \mathbf{r}](z) \cdot r_a(\sqrt{\hbar} z,\mathbf{t},\sqrt{\hbar}\mathbf{r})  \Big] (-\vartheta(\mathbf{r},\mathbf{t})) & \\[0.2cm]
 & \hspace*{-6cm} = i \int_{\R^{2d}} \nabla \widehat{\mathfrak{H}}_{\nu,\nu'}[\mathbf{t}, \mathbf{t} + \sqrt{\hbar} \mathbf{r}] (w + \vartheta(\mathbf{r},\mathbf{t})) \cdot \frac{1}{\hbar^d} \mathcal{F}_z [ r_a( z, \mathbf{t}, \sqrt{\hbar} \mathbf{r})]\left( \frac{w}{\sqrt{\hbar}} \right) dw,
\end{align*}
hence, denoting $r = s - s'$,
\begin{align*}
\int_{\R^{d_E+1}} \chi_T(s)\chi_T( s + \sqrt{\hbar} r) \left \vert \mathcal{F}_z \Big[ z \mathfrak{H}_{\nu,\nu'}[\mathbf{t}, \mathbf{t} + \sqrt{\hbar} \mathbf{r}](z) \cdot r_a(\sqrt{\hbar} z,\mathbf{t},\sqrt{\hbar}\mathbf{r})  \Big] (\vartheta(\mathbf{r},\mathbf{t})) \right \vert d\mathbf{r}  \\[0.2cm]
 & \hspace*{-12cm} \leq \sum_{\vert \gamma \vert = 1}  \sup_{\tilde{\mathbf{r}} \in \Omega_{d_E} \times [-T,T]}\big \Vert \mathcal{F}_z[ r_a^\gamma(z, \mathbf{t},\tilde{\mathbf{r}})](\cdot) \big \Vert_{L^1(\R^{2d})} \int_{\R^{d_E+1}} \big \vert \partial^\gamma \widehat{ \mathfrak{H}}_{\nu,\nu'}[\mathbf{t},\mathbf{t}+ \sqrt{\hbar} \mathbf{r}](\vartheta(\mathbf{r},\mathbf{t})) \big \vert d \mathbf{r},
\end{align*}
where for any $\gamma \in \mathbb{N}^{2d}$, $r_a^\gamma =  \int_0^1 \partial^\gamma a( z(\mathbf{t}) + s (\sqrt{\hbar} z +O_{\mathbf{t}}( \sqrt{\hbar} \vert \mathbf{r} \vert))) ds$. This argument shows that, using \eqref{e:constant_estimate}, \eqref{e:estimate_norm__fourier_hermite} together with the dominated convergence theorem, one has
\begin{align*}
\frac{C_T(z_0)}{\sqrt{\hbar^{d_E+1}}} \int_{(\mathbb{T}^{d_E} \times \R)^2} \int_{\R^{2d}}  \kappa_{\nu,\nu'}^T(\mathbf{t}, \mathbf{t}') a\big( \sqrt{\hbar} z + \mathbf{z}(\mathbf{t},\mathbf{t}') \big) e^{\frac{i}{\hbar} f_1(\mathbf{t},\mathbf{t}')} e^{\frac{i}{\sqrt{\hbar}} f_2(z,\mathbf{t},\mathbf{t}')} \mathfrak{H}_{\nu,\nu'}[\mathbf{t}, \mathbf{t}'](z) dz d \mathbf{t} d \mathbf{t}' \\[0.2cm]
 & \hspace*{-14.8cm} = C_T(z_0) \int_{\mathbb{T}^{d_E} \times \R} \int_{\R^{d_E+1}}  \kappa_{\nu,\nu'}^T(\mathbf{t}, \mathbf{t}) a(z(\mathbf{t})) \widehat{\mathfrak{H}}_{\nu,\nu'}[\mathbf{t}, \mathbf{t}]( -\vartheta(\mathbf{r},\mathbf{t})) \frac{d\mathbf{r}}{ \vert \mathbb{T}^{d_E}  \vert} d\mathbf{t} +  O\left( \hbar^{\frac{1}{2} + \frac{\vert \nu \vert + \vert \nu' \vert}{2}} \right) D_T,
\end{align*}
where, by \eqref{e:estimate_norm__fourier_hermite}, $|D_T|\leq C \theta(z_0,T)^{\vert \nu \vert + \vert \nu'\vert }$.

In particular, by \cite[(11.44)]{DeGosson11}, we can compute the explicit expressions for the Gaussian $\mathfrak{H}_{0,0}[\mathbf{t},\mathbf{t}]$:
$$
\mathfrak{H}_{0,0}[\mathbf{t}, \mathbf{t}](z) =  \frac{1}{\pi^d} e^{- (F_{z_0}(\mathbf{t})^{-1})^{\textnormal{t}} F_{z_0}(\mathbf{t})^{-1} z \cdot z},
$$
and hence, the Fourier inversion formula yields:
$$
\widehat{\mathfrak{H}}_{0,0}[\mathbf{t}, \mathbf{t}](w) = \frac{1}{\pi^d} e^{- F_{z_0}(\mathbf{t}) F_{z_0}(\mathbf{t})^{\textnormal{t}} w \cdot w}.
$$
Using next that $F^{\textnormal{t}}_{z_0} = JF^{-1}_{z_0}J^{-1}$, since $F_{z_0}$ is a symplectic matrix, we also have that 
\begin{align*}
F^{\textnormal{t}}_{z_0}(\mathbf{t}) J X_{\widetilde{\cH}_j^E}(z(\mathbf{t})) & = J F_{z_0}(\mathbf{t})^{-1} X_{\widetilde{\cH}_j^E}(z(\mathbf{t})) = J X_{\widetilde{\cH}_j^E}(z_0) = d\widetilde{\cH}_j^E(z_0), \\[0.2cm]
F^{\textnormal{t}}_{z_0}(\mathbf{t}) J X_{\langle L \rangle}(z(\mathbf{t})) & = J F_{z_0}(\mathbf{t})^{-1} X_{\langle L \rangle}(z(\mathbf{t})) = J X_{\langle L \rangle}(z_0) = d\langle L \rangle(z_0).
\end{align*}
Therefore:
\begin{align*}
C_T(z_0) \int_{\mathbb{T}^{d_E} \times \R} \int_{\R^{d_E+1}} \chi_T(s)^2 a(z(\mathbf{t})) \widehat{\mathfrak{H}}_{0,0}[\mathbf{t}, \mathbf{t}]( \vartheta(\mathbf{r},\mathbf{t})) \frac{d\mathbf{r}}{ \vert \mathbb{T}^{d_E}  \vert} d\mathbf{t} & \\[0.2cm]
 & \hspace*{-6cm} = C_T(z_0) \int_{\mathbb{T}^{d_E} \times \R} \int_{\R^{d_E+1}}  \chi_T(s)^2 a(z(\mathbf{t})) \widehat{\mathfrak{H}}_{0,0}[\mathbf{0}, \mathbf{0}]\big( \vartheta(\mathbf{r},0) \big) \frac{d\mathbf{r}}{ \vert \mathbb{T}^{d_E}  \vert} d\mathbf{t} \\[0.2cm]
 & \hspace*{-6cm} = \frac{1}{\Vert \chi_T \Vert_{L^2}^2 }   \int_{\mathbb{T}^{d_E} \times \R} \chi_T(s)^2 a(z(\mathbf{t})) d\mathbf{t} \\[0.2cm]
 & \hspace*{-6cm} = \frac{1}{\Vert \chi_T \Vert_{L^2}^2} \int_\R \chi_T(s)^2 \, \langle a \rangle \circ \phi_s^{\langle L \rangle}(z_0) ds,
\end{align*}
and, moreover,
\begin{align*}
\left \vert C_T(z_0) \int_{\mathbb{T}^{d_E} \times \R} \int_{\R^{d_E+1}} \kappa_{\nu,\nu',\hbar}^T(\mathbf{t}, \mathbf{t}) a(z(\mathbf{t})) \widehat{\mathfrak{H}}_{\nu,\nu'}[\mathbf{t}, \mathbf{t}]( \vartheta(\mathbf{r},\mathbf{t})) d\mathbf{r} d\mathbf{t}  \right \vert = O\left( \hbar^{ \frac{\vert \nu \vert + \vert \nu' \vert}{2}} \theta(z_0,T)^{\vert \nu \vert + \vert \nu' \vert} \right).
\end{align*}

Finally, we prove that the the sequence $\psi_\hbar$ satisfies the condition of being a quasimode of the required width $O(\delta_\hbar \hbar)$ by choosing $l$ sufficiently large. Using the definition \eqref{e:quasimode_I} of $\psi_\hbar$ and integration by parts in $\tau = (\tau_1, \ldots, \tau_{d_E})$ and in $s$, we obtain, defining $\lambda_\hbar := \tilde{\lambda}_\hbar^1 + \cdots + \tilde{\lambda}_\hbar^{d_E} + \delta_\hbar E'$, that
\begin{align*}
(\widehat{H}_\hbar + \delta_\hbar \langle \widehat{L}_\hbar \rangle  - \lambda_\hbar ) \psi_\hbar &  \\[0.2cm]
 & \hspace{-3cm}
  =  -i \delta_\hbar \hbar \left(\frac{C_T(z_0)}{\sqrt{\hbar^{d_E+1}}} \right)^{1/2} \; \sum_{\vert \nu \vert = 0}^{3l - 1} \int_{\mathbb{T}^{d_E} \times \R} \chi'_T(s) e^{\frac{i(M_\hbar,E')\cdot (\tau,s)}{\hbar}} c_\nu(\tau,s,\hbar) \psi_\nu^\hbar(\tau,s,x)  \frac{d\tau}{|\T^{d_E}|}ds \\[0.2cm]
&  \hspace{-3cm} \quad - i \delta_\hbar \hbar \left(\frac{C_T(z_0)}{\sqrt{\hbar^{d_E+1}}} \right)^{1/2} \;  \int_{\mathbb{T}^{d_E} \times \R} \chi'_T(s) \mathcal{R}_l(\cdot,z_0,\tau,s,\hbar)  \frac{d\tau}{|\T^{d_E}|}ds \\[0.2cm]
 & \hspace{-3cm} \quad -  \left(\frac{C_T(z_0)}{ \sqrt{\hbar^{d_E+1}}}
\right)^{1/2} \int_{\mathbb{T}^{d_E} \times \R} \chi_T(s) (\widehat{H}_\hbar + \delta_\hbar \langle \widehat{L}_\hbar \rangle ) \mathcal{R}_l(\cdot,z_0,\tau,s,\hbar)   \frac{d\tau}{|\T^{d_E}|}ds \\[0.2cm]
& \hspace{-3cm} =: R_\hbar^1 + R_\hbar^2 + R_\hbar^3.
\end{align*}
Using that that $\operatorname{WF}_\hbar\big(  \mathcal{R}_l(\cdot,z_0,\tau,s,\hbar) \big) = \{ z(\tau,s) \}$ and \eqref{e:reminder_estimate_ehrenfest} we get that 
\begin{align}
\label{e:remainder_l}
\Vert R_\hbar^3 \Vert_{L^2} & = O\left(  \hbar^{-\frac{d_E + 1}{4} + \frac{l}{2}} T \theta(z_0,T)^{2+l} \right) \\[0.2cm]
\Vert R_\hbar^2 \Vert_{L^2} & = O\left( \delta_\hbar \hbar^{-\frac{d_E + 1}{4} + \frac{l}{2} + 1} T \theta(z_0,T)^{2+l} \right).
\end{align}
Taking $l$ sufficiently large (in particular $l \geq d_E + 2$) is enough to bound these two terms by an $O(\sqrt{\hbar}\theta(z_0,T))$ for $0 < T < \mathscr{T}_\hbar$.
We finally estimate $\Vert R_\hbar^1 \Vert_{L^2}$ by repeating the previous argument used to compute $W_{\psi_\hbar}^\hbar(a)$, but replacing $\chi_T$ by $\chi_T^\prime$ in the integrand and taking $a \equiv 1$, and we obtain
\begin{equation}
\Vert R^1_\hbar \Vert_{L^2} = c(T,\hbar) \delta_\hbar \hbar, 
\end{equation}
where $c(T,\hbar)$ is given by \eqref{e:constant_c} with $C_\chi = \frac{\Vert \chi' \Vert_{L^2}}{\Vert \chi \Vert_{L^2}}$.

Assume now that $\Gamma_T = \{ 0 \}$. In this case, notice that 
$$
\mathbb{T}^{d_E} \times \R \in (\tau,s) \mapsto \phi_s^{\langle L \rangle} \circ \Phi^{\widetilde{\cH}^E}_\tau(z_0)  = \Phi^{\widetilde{\cH}^E}_{\tau + v_1 s}(z_0) \in \mathcal{C}_{z_0}(T),
$$
where $v_1 \in \mathbb{T}^{d_E}$. Moreover,
$$
F_{z_0}(\tau,s) := d \big( \phi_s^{\langle L \rangle} \circ \Phi^{\widetilde{\cH}^E}_\tau(z_0) \big) = d \Phi_{\tau + v_1 s}^{\widetilde{\cH}^E}(z_0) =: \tilde{F}_{z_0}(\tau + v_1 s).
$$
We consider again the unitary operator given by \eqref{e:unitary_operator}. We define $\psi_\hbar$ by  \eqref{e:quasimode_I}, where now the normalizing constant $C_T(z_0)$ is given by
$$
C_T(z_0) = \sqrt{ \frac{\det \widetilde{\mathcal{G}}_{z_0}}{\pi^{d_E} }} \cdot \frac{ \vert \mathbb{T}^{d_E}   \vert}{T^2 \Vert \chi \Vert_{L^1}^2 }, \quad \widetilde{\mathcal{G}}_{z_0} :=   d \widetilde{\mathcal{H}}^E(z_0)  \,  d \widetilde{\mathcal{H}}^E(z_0)^{\textnormal{t}},
$$
and observe that
\begin{align*}
\psi_\nu^\hbar(\tau,s,x) & = e^{\frac{i \gamma(\tau,s)}{\hbar}} \, \widehat{T}\big( \phi_s^{\langle V \rangle} \circ \Phi^{\widetilde{\cH}^E}_\tau(z_0) \big) \,\widehat{F}_{z_0}(\tau,s) \Psi_\nu^\hbar(x) \\[0.2cm]
 & = e^{\frac{i \gamma(\tau,s)}{\hbar}} \, \widehat{T}\big(  \Phi_{\tau + v_1 s}^{\widetilde{\cH}^E}(z_0) \big) \,\widehat{\tilde{F}}_{z_0}(\tau + v_1 s) \Psi_\nu^\hbar(x).
\end{align*}
Then using that $\vert c_0(\tau,s,\hbar) - 1 \vert \leq C_l \hbar^{1/2}$ and \eqref{e:constant_estimate}, and performing the change of variables $\widetilde{\tau} = \tau + v_1 s$, $\widetilde{\tau}' = \tau' + v_1 s'$ inside the oscillatory integral, which now has stationary phase at\footnote{Notice, regarding the parameters $(s,s') \in [0,T]\times [0,T]$ that there is no concentration on the diagonal $s = s'$. In other words, the metric of the Gaussian in the variable $\mathbf{r} = \widetilde{\tau} - \widetilde{\tau}'$ has now rank $d_E$ (and not $d_E+1$) coinciding with the rank of $d \widetilde{\mathcal{H}}^E(z_0)$.} $\widetilde{\tau} = \widetilde{\tau}'$ we obtain the same limit estimate as before. Indeed, since $\langle a \rangle \circ \phi_s^{\langle L \rangle}(z_0) = \langle a \rangle(z_0)$, we have
$$
\langle a \rangle (z_0) = \frac{1}{ \Vert \chi_T \Vert_{L^2(\R)}^2} \int_\R \chi_T(s)^2 \, \langle a \rangle \circ \phi_s^{\langle L \rangle}(z_0) ds.
$$
The proof then follows essentially the same lines of the previous one. The constant $C_\chi$ is in this case given by $C_\chi = \frac{\Vert \chi' \Vert_{L^1}}{\Vert \chi \Vert_{L^1}}$.
\end{proof}

\section{Proofs of Theorems \ref{t:quasimodes_1} and \ref{t:improved_quasimodes}}  \label{s:proofs}

In this final section, we construct quasimodes for the perturbed operator $\widehat{P}_\hbar$, and prove the concentration and non-concetration properties displayed by the quasimodes of this system with different widths. 

\begin{proof}[Proof of Theorem \ref{t:quasimodes_1}]
We first assume that $\varepsilon_\hbar = O(\hbar)$. Using the normal form \eqref{conjugation-1}, we have
$$
\widehat{P}_\hbar^1 = \widehat{H}_\hbar + \varepsilon_\hbar \langle \widehat{V}_\hbar \rangle + O(\varepsilon_\hbar^{2}),
$$ 
Using Lemma \ref{concentration_minimal_set_2} with $L = V$ and $\delta_\hbar = \varepsilon_\hbar$, defining $\psi_\hbar^1(x) = U_{1,\hbar} \psi_\hbar$, where $\psi_\hbar$ is given by \eqref{e:quasimode_I}, and using that $\varepsilon_\hbar = O(\hbar)$, we have
$$
\widehat{P}_\hbar \psi_\hbar^1 = \lambda_\hbar \psi_\hbar^1 + O(\varepsilon_\hbar^2) + O(\varepsilon_\hbar \hbar) = \lambda_\hbar \psi_\hbar^1 +  O(\varepsilon_\hbar \hbar).
$$
Moreover, by \eqref{FIO-estimate} and \eqref{e:best_concentration_2}, we obtain
\begin{equation}
\label{e:concentration_sharpened}
\big \langle \psi^1_\hbar, \Op_\hbar(a) \psi_\hbar^1 \big \rangle_{L^2(\R^d)} = \frac{1}{ \Vert \chi_T \Vert_{L^2(\R)}^2} \int_\R \chi_T(s)^2 \, \langle a \rangle \circ \phi_s^{\langle V \rangle}(z_0) ds + O_T(\hbar^{1/2}).
\end{equation}
In particular,
\begin{equation}
\label{e:measure_sharpened}
\mu_T = \frac{1}{ \Vert \chi_T \Vert_{L^2(\R)}^2} \int_\R \chi_T(s)^2 \, \big( \phi_s^{\langle V \rangle} \big)_* \mu_{\omega,z_0}ds  \in \mathcal{Q}(\widehat{P}_\hbar, \varepsilon_\hbar \hbar).
\end{equation}
In the general case $\varepsilon_\hbar = O(\hbar^\alpha)$, we need to use the normal form up to higher order. We take $N$ such that $\varepsilon_\hbar^{N} = O(\hbar)$. Then we use Lemma \ref{concentration_minimal_set_2} to construct a quasimode $(\psi_\hbar)$ of width $O(\varepsilon_\hbar \hbar)$ for the truncated operator
\begin{align}
\label{e:up_to_N}
\widehat{H}_\hbar + \widehat{\mathcal{P}}_{N,\hbar} := \widehat{H}_\hbar + \varepsilon_\hbar \langle \widehat{V}_\hbar \rangle + \sum_{j=2}^N \varepsilon_\hbar^j \langle \widehat{R}_{j,\hbar} \rangle.
\end{align}
The construction is completely analogous to the one given in the proof of Lemma \ref{concentration_minimal_set_2}, up to replacing the flow $\phi_s^{\langle V \rangle}$ by the flow generated by the symbol of $\widehat{\mathcal{P}}_{N,\hbar}$, and adapting all the estimates for $\phi_s^{\langle L \rangle}$ to this new Hamiltonian system commuting with $H$. Moreover, we replace the quasi-eigenvalue $E' = \langle V \rangle(z_0)$ by
$$
E_\hbar = \langle V \rangle(z_0) + \sum_{j=2}^N \varepsilon_\hbar^{j-1} \langle R_{j}(\hbar) \rangle(z_0).
$$
Thus, defining $\psi_\hbar^N(x) = U_{N,\hbar} \psi_\hbar$, we get
$$
\widehat{P}_\hbar \psi_\hbar^N = \lambda_\hbar \psi_\hbar^N + O(\varepsilon_\hbar^{N+1}) + O(\varepsilon_\hbar \hbar) = \lambda_\hbar \psi_\hbar^N +  O(\varepsilon_\hbar \hbar).
$$
Moreover, by \eqref{e:best_concentration_2} and \eqref{FIO-estimate}, we obtain again \eqref{e:concentration_sharpened} and \eqref{e:measure_sharpened}.

It remains to use a convexity argument to finish the proof. Let $\mu_0 \in \mathcal{M}(\widehat{H}_\hbar)$ with $\supp \mu_0 \in \langle V \rangle^{-1}(E')$. By the Krein-Milman theorem, $\mu_0$ is in the convex hull, for the weak-$\star$ topology, of the set of measures of the form
$$
\mu = \sum_{j=1}^N \alpha_j \mu_{z_j,\omega},
$$
where $z_j = z_j \in H^{-1}(1) \cap \langle V \rangle^{-1}(E')$, $\mathcal{T}_\omega(z_j) \cap \mathcal{T}_\omega(z_k) = \emptyset$ for every $k \neq j$, the limit is taken in the weak-$\star$ sense, $0\leq \alpha_j \leq 1$, and $\sum_{j=1}^{N} \alpha_j = 1$. We have shown that, for every $1 \leq j \leq N$, there exists a quasimode $(\psi_{j,\hbar})$ of width $O(\varepsilon_\hbar \hbar)$ such that
$$
\big \langle \psi_{j,\hbar}, \Op_\hbar(a) \psi_{j,\hbar} \big \rangle_{L^2(\R^d)} = \frac{1}{ \Vert \chi_T \Vert_{L^2(\R)}^2} \int_\R \chi_T(s)^2 \, \langle a \rangle \circ \phi_s^{\langle V \rangle}(z_j) ds + o(1),
$$
and such that
$$
\big( \widehat{H}_\hbar + \varepsilon_\hbar  V_\hbar  \big) \psi_{j,\hbar}(x) = \lambda_\hbar \psi_{j,\hbar}(x) + O(\varepsilon_\hbar \hbar).
$$
Defining the sequence\footnote{Note that the sequence $(\varphi_\hbar^N)$ is  asymptotically normalized.}
$$
\varphi^N_{\hbar} = \sum_{j=1}^N \frac{1}{\sqrt{\alpha_j}} \psi_{j,\hbar}, \quad N \in \mathbb{N},
$$
and using \cite[Proposition 3.3]{Ger90} to calculate the cross-Wigner distributions of terms with disjoint weak-$\star$ limits, we get
\begin{align*}
\big \langle \varphi^N_{\hbar}, \Op_\hbar(a) \varphi^N_{\hbar} \big \rangle_{L^2(\R^d)} & = \frac{1}{ \Vert \chi_T \Vert_{L^2(\R)}^2} \int_\R \chi_T(s)^2 \, \sum_{j=1}^N \int_{\mathcal{T}_\omega(z_j)} \langle a \rangle \circ \phi_s^{\langle V \rangle}(z) \mu_{\omega,z_j}(dz) ds + o(1) \\[0.2cm]
& = \frac{1}{ \Vert \chi_T \Vert_{L^2(\R)}^2} \int_\R \int_{\R^{2d}} \chi_T(s)^2 \langle a \rangle \circ \phi_s^{\langle V \rangle}(z) \mu(dz) ds + o(1),
\end{align*}
and that
$$
\big( \widehat{H}_\hbar + \varepsilon_\hbar  V_\hbar  \big) \varphi^N_{\hbar}(x) = \lambda_\hbar \varphi^N_{\hbar}(x) + O(\varepsilon_\hbar \hbar).
$$
Therefore, by a diagonal extraction argument, we can define a quasimode $\psi_\hbar := \varphi_\hbar^{N_\hbar}$, for $N_\hbar \to \infty$ as $\hbar \to 0^+$, of width $O(\varepsilon_\hbar \hbar)$, with semiclassical measure given by \eqref{e:evolved_measure}. 
\end{proof}

\begin{proof}[Proof of Theorem \ref{t:improved_quasimodes}]
We just have to modify a bit the proof of Theorem \ref{t:quasimodes_1}. Take a continuous family of bump functions $(\chi_T) \subset \mathcal{C}_c^\infty(\R)$ for $T \geq T_0 > 0$ given by $\chi_T = \chi( \cdot/T)$. Let $(T_\hbar)$ be a sequence of parameters such that $T_\hbar \to +\infty$ sufficiently slowly (below the Ehrenfest time $\mathscr{T}_\hbar$ up to which the approximation \eqref{e:linear_combination_of_coherent_states} is valid). By the proof of Theorem \ref{t:quasimodes_1}, given $\mu_0 \in \mathcal{M}(\widehat{H}_\hbar)$ with $\supp \mu_0 \subset \mathcal{I}(H, \langle V \rangle)$, we can construct a sequence $(\psi_\hbar^N)$ with $\Vert \psi_\hbar^N \Vert_{L^2} = 1$ such that
\begin{equation}
\label{e:estimate_one}
\big \langle \psi_\hbar^N, \Op_\hbar(a) \psi_\hbar^N \big \rangle_{L^2(\R^d)} = \frac{1}{ \Vert \chi_{T_\hbar} \Vert_{L^2(\R)}^2} \int_\R \int_{\R^{2d}} \chi_{T_\hbar}(s)^2 \, \langle a \rangle \circ \phi_s^{\langle V \rangle} (z_0) \mu_0(dz) ds + o(1),
\end{equation}
and
\begin{equation}
\label{e:estimate_two}
\widehat{P}_\hbar \psi^N_\hbar = \lambda_\hbar \psi^N_\hbar + O(\varepsilon_\hbar^{N+1}) + c(T_\hbar,\hbar)  \varepsilon_\hbar \hbar.
\end{equation}
Since $\varepsilon_\hbar = O(\hbar^\alpha)$, we can choose $N$ sufficiently large so that $\varepsilon_\hbar^N = o(\hbar)$. Taking the limit $\hbar \to 0^+$ in \eqref{e:estimate_one} and \eqref{e:estimate_two}, we conclude that 
$$
\mathcal{A}_{\langle V \rangle}(\mu_0)  \in \mathcal{Q}(\widehat{P}_\hbar, r_\hbar),
$$
with $r_\hbar = o(\varepsilon_\hbar \hbar)$.

It remains to show the second assertion of the statement of Theorem \ref{t:improved_quasimodes} 
when $\varepsilon_\hbar \ll \hbar^{1 + \gamma(\omega)}$. Let $\mu \in \mathcal{Q}(\widehat{P}_\hbar,r_\hbar)$. By Theorem \ref{t:invariance_for_quasimodes}, it is enough to show that there exist $\cE \in\Sigma_\cH$ and $E \in \R$ such that
$$
\supp \mu \subset \mathcal{H}^{-1}(\cE) \cap \langle V \rangle^{-1}(E).
$$
We reason as follows. Using the normal form \eqref{quantum_normal_form}, we have
$$
\widehat{P}_\hbar^N = U_{N,\hbar}^*\big( \widehat{H}_\hbar + \varepsilon_\hbar \widehat{V}_\hbar \big) U_{N,\hbar} = \widehat{H}_\hbar + \varepsilon_\hbar \langle \widehat{V}_\hbar \rangle + \sum_{j=2}^N \varepsilon_\hbar^j \langle \widehat{R}_{j,\hbar} \rangle +  O(\varepsilon_\hbar^{N+1}).
$$
Then the sequence given by $\psi_\hbar^N = (U_{N,\hbar}^* \psi_\hbar)$ is a quasimode for the operator $\widehat{H}_\hbar + \widehat{\mathcal{P}}_{N,\hbar}$ given by \eqref{e:up_to_N} of width $o(\varepsilon_\hbar \hbar)$. 

By Lemma \ref{l:projection_lemma},  the mass of the sequence $(\psi_\hbar^N)$ is concentrated around the eigenspace of $\lambda_\hbar$.  Proposition \ref{p:localization} then shows that
every semiclassical measure $\mu$ associated with the sequence $(\psi_\hbar^N)$, satisfies $\supp \mu \subset \mathcal{H}^{-1}(\cE) $ for some $\cE \in \Sigma_\cH$.
Moreover, we also have that
$$
\langle \widehat{V}_\hbar \rangle \psi_\hbar^N = \frac{\Lambda_\hbar - \lambda_\hbar}{\varepsilon_\hbar} \psi^N_\hbar + o(1).
$$ 
Then there exists $E$ such that $(\Lambda_\hbar - \lambda_\hbar)/\varepsilon_\hbar \to E$, and every semiclassical measure associated with the sequence $(\psi^N_\hbar)$ is supported on $\langle V \rangle^{-1}(E)$. Since the sequences $(\psi^N_\hbar)$ and $(\psi_\hbar)$ share the same semiclassical measure by \eqref{FIO-estimate}, the proof is complete. 
\end{proof}

Following the same lines of the previous proof, we get the following converse for Theorem \ref{t:general_invariance}:

\begin{corol}
\label{c:corol_general_invariance}
Let us assume the hypothesis of  Theorem \ref{t:general_invariance}. If $H^{-1}(1) = \mathcal{I}(H,\langle L \rangle)$ and $kj > 1+\gamma(\omega)$, then for every sequence $(r_\hbar) \subset \R_+$ satisfying $r_\hbar = o(\hbar^{kj +1})$, every $\mu \in \mathcal{Q}(\widehat{P}_\hbar, r_\hbar)$ satisfies:
$$
\supp \mu \subset \mathcal{H}^{-1}(\cE)  \cap \langle L \rangle^{-1}(E),
$$
for some $\cE \in \Sigma_\cH$ and some $E \in \R$.
\end{corol}

\appendix

\section{A Quantum Birkhoff Normal Form}
\label{averaging_method}

The averaging method and quantum Birkhoff normal forms are widely used in classical and quantum mechanics (see \cite{Arnold89,Ver96,Gui81,Mos70,Ur85, UV08,WeinsteinZoll}) when it comes to the study of perturbations of completely integrable systems. Roughly speaking, the main idea consists in averaging the perturbation along the orbits of the completely integrable system to a given order, in such a way that the perturbed Hamiltonian, up to this order, commutes with the unperturbed one. 

In the particular case of quantum harmonic oscillators, two references particularly close to our purposes are \cite{CharVNgoc} for the homogeneous periodic oscillator, and \cite{Sj92} for the non-resonant case. In the context of reducibility of polynomial time-dependent perturbations of the harmonic oscillator, see also \cite{Bam17}.

Here we need a variant of this type of result, that covers all intermediate cases. Since we work in the semiclassical limit as the eginvalues of $\widehat{P}_\hbar$ verify $\lambda_\hbar \to 1$, then our quasimodes concentrate in phase-space on the level set $H^{-1}(1)$. This means that we can not use Taylor expansions of our symbols near zero, and then we need to use a Diophantine hypothesis to solve the cohomological equations which will appear at each step of the normal-form construction.

\subsection{Cohomological equations}\label{cohomological-equations}

One of the technical difficulties that we will find in the process of averaging the perturbation $V$ by the classical flow of the harmonic oscillator, will be to deal with cohomological equations \cite[Sec. 2.5]{Llav03} as the following:
\begin{equation}
\label{cohomological}
\{ H , f \} = g,
\end{equation}
where $g \in \mathcal{C}^\infty(\R^{2d})$ is a smooth function such that $\langle g \rangle = 0$. The goal is to solve this equation in $\mathcal{C}^\infty(\R^{2d})$. 

Write for $f \in \mathcal{C}^\infty(\R^{2d})$:
\begin{equation}
\label{e:harmonic_fourier-decomposition}
f\circ\Phi_\tau^\cH=\frac{1}{(2\pi)^{d_\omega}}\sum_{k \in \mathbb{Z}^{d_\omega}}f_k e^{ik\cdot\tau}  ,\quad  f_k:=\int_{\T^{d_\omega}} f \circ\Phi_\tau^\cH e^{-ik \cdot \tau} d\tau \; \text{ for }\; k\in\IZ^{d_\omega}
\end{equation}
Observe that if $f$ is a solution to (\ref{cohomological}), then so is $f + \lambda \langle f \rangle$ for any $\lambda \in \mathbb{R}$, since $\{ H, \langle f \rangle \} = 0$. Thus we can try to solve the equation for $\langle f \rangle = 0$ fixed, imposing $\langle f \rangle=0$.
Writing down
$$
\{ H, f \} = \frac{d}{dt} \left(f \circ \Phi_{tv} \right) \vert_{t=0} = \frac{1}{(2\pi)^{d_\omega}} \sum_{k \in \mathbb{Z}^{d_\omega} \setminus \{0\}} (ik \cdot v) \, f_k =  \frac{1}{(2\pi)^{d_\omega}} \sum_{k \in \mathbb{Z}^d} g_k,
$$
we obtain that the solution of (\ref{cohomological}) is given (at least formally) by
\begin{equation}
\label{solution-cohomological}
f = \frac{1}{(2\pi)^{d_\omega}} \sum_{k \in \mathbb{Z}^{d_\omega} \setminus \{0\}} \frac{1}{ik\cdot v} \, g_k.
\end{equation}
It is not difficult to see that, unless we impose some quantitive restriction on how fast $\vert vk \cdot v \vert^{-1}$ can grow, the solutions given formally by (\ref{solution-cohomological}) may fail to be even distributions (see for instance \cite[Ex. 2.16.]{Llav03}). But if $v$ is Diophantine (which is always the case when $\omega$ is partially Diophantine \eqref{e:partially_diophantine2}), 
and $g \in \mathcal{C}^\infty(\R^{2d})$ is such that $\langle g \rangle = 0$, then \textnormal{(\ref{solution-cohomological})} defines a smooth solution $f \in \mathcal{C}^\infty(\R^{2d})$ of \textnormal{(\ref{cohomological})}.

Finally, in the periodic case (assuming $\omega = (1, \ldots , 1)$ for simplicity), the solution to the cohomological equation \textnormal{(\ref{cohomological})} is given by the explicit formula
\begin{equation}
\label{e:solution_periodic_case}
f = \frac{-1}{2\pi} \int_0^{2\pi} \int_0^t  g \circ \phi_s^H  \, ds \, dt,
\end{equation}
provided that $\langle f \rangle = \langle g \rangle = 0$.

\subsection{Quantum Birkhoff normal form}
We finally give the normal form for the operator $\widehat{H}_\hbar + \varepsilon_\hbar \widehat{V}_\hbar$ assuming that $\omega$ is partially Diophantine.
\begin{propop}
\label{averaging-lemma}
If $\omega$ is partially Diophantine then, for every $N \in \mathbb{N}$, there exists a sequence of unitary operators $(U_{N,\hbar})$ on $L^2(\R^d)$ such that 
\begin{equation}
\label{FIO-estimate}
\left \Vert U_{N,\hbar} \Op_\hbar(a) U_{N,\hbar}^* - \Op_\hbar(a) \right \Vert_{\mathcal{L}(L^2)} = O_{\mathcal{L}(L^2)}(\varepsilon_\hbar), \quad \text{for all } a \in \mathcal{C}_c^\infty(\R^{2d}),
\end{equation}
and
\begin{equation}
\label{quantum_normal_form}
\widehat{P}_\hbar^N := U_{N,\hbar}^* \big( \widehat{H}_\hbar + \varepsilon_\hbar \widehat{V}_\hbar \big) U_{N,\hbar} =  \widehat{H}_\hbar + \sum_{j=1}^N \varepsilon_\hbar^j \langle \widehat{R}_{j,\hbar} \rangle + O_{\mathcal{L}(L^2)}(\varepsilon_\hbar^{N+1}),
\end{equation}
where $\widehat{R}_{1,\hbar} = \widehat{V}_\hbar$, and $\widehat{R}_{j,\hbar}$ are $L^2$-bounded semiclassical pseudodifferential operators that do not depend on $N$.
\end{propop}
\begin{remark}
This statement can be simplified when $\varepsilon_\hbar=h^k$ with $k\geq 1$ an integer.
Formula \eqref{quantum_normal_form} then reads:
\begin{equation}\label{e:int_qbnf}
    U_{N,\hbar}^* \big( \widehat{H}_\hbar + \varepsilon_\hbar \widehat{V}_\hbar \big) U_{N,\hbar} =  \widehat{H}_\hbar + \hbar^k\sum_{j=0}^{k(N-1)} \hbar^j\langle \Op_{\hbar}(r_j) \rangle + O_{\mathcal{L}(L^2)}(\hbar^{kN+1}),
\end{equation}
for some $r_j\in S^0(\R^{2d})$, $j=0,\hdots k(N-1)$, with $r_0=V$.
If, in addition, $\omega$ is Diophantine (in particular, $d_\omega = d$), then \eqref{quantum_normal_form} can also be written as:
\begin{equation}
\label{e:normal_form_diphantine_case}
    U_{N,\hbar}^* \big( \widehat{H}_\hbar + \varepsilon_\hbar \widehat{V}_\hbar \big) U_{N,\hbar} =  \widehat{H}_\hbar + \hbar^k\sum_{j=0}^{kN-1} \hbar^j g_j(\Op_\hbar(H_{1}),\hdots, \Op_\hbar(H_d)) + O_{\mathcal{L}(L^2)}(\hbar^{k(N+1)}),
\end{equation}
for some $g_j\in S^0(\R^d)$, $j=0,\hdots, k(N-1)$.
\end{remark}
\begin{proof}
We fix $N \geq 1$. Let $F_1 \in S^0(\R^{2d})$, and denote its Weyl quantization by $\widehat{F}_{1,\hbar} = \Op_\hbar(F_1)$. We define the following unitary operator:
\begin{equation}
\label{FIO-1}
\mathcal{U}_1(t) := \exp \left[ -\frac{it\varepsilon_\hbar}{\hbar}  \widehat{F}_{1,\hbar} \right] =  \sum_{j=0}^\infty \frac{1}{j!} \left(-\frac{it\varepsilon_\hbar}{\hbar}  \widehat{F}_{1,\hbar}\right)^j, \quad t \in [0,1],
\end{equation}
where the series converges in the $\mathcal{L}(L^2)$-norm provided that $\widehat{F}_{1,\hbar}$ is a bounded operator on $L^2(\R^d)$. We denote $\mathcal{U}_1 = \mathcal{U}_1(1)$ and
conjugate $\widehat{P}_\hbar = \widehat{P}^0_\hbar := \widehat{H}_\hbar + \varepsilon_\hbar \widehat{V}_\hbar$ by $\mathcal{U}_1$, obtaining:
\begin{align*}
\widehat{P}_\hbar^1 := \mathcal{U}_1^*  \widehat{P}^0_\hbar \, \mathcal{U}_1 & = \widehat{H}_\hbar + \varepsilon_\hbar \widehat{V}_\hbar + \sum_{j=1}^N \frac{\varepsilon_\hbar^j}{j!} \left(\frac{i}{\hbar} \right)^j  \Ad^j_{\widehat{F}_{1,\hbar}} (\widehat{H}_\hbar) \\[0.2cm]
 & \quad + \sum_{j=1}^{N-1} \frac{\varepsilon_\hbar^{j+1}}{j!} \left(\frac{i}{\hbar} \right)^j  \Ad^j_{\widehat{F}_{1,\hbar}} (\widehat{V}_\hbar) + \varepsilon_\hbar^{N+1} \widehat{T}_\hbar,
\end{align*}
where $\Ad^j_P(Q) := [P, \Ad^{j-1}_P(Q)]$, and $\Ad^0_P(Q) = Q$. The Taylor remainder $\widehat{T}_\hbar$ is given by
\begin{align*}
\widehat{T}_\hbar &  = \int_0^1 \frac{(1-t)^N}{N!} \left( \frac{i}{\hbar} \right)^{N+1} \mathcal{U}_1(t)^* \Ad^{N+1}_{\widehat{F}_{1,\hbar}}(  \widehat{H}_\hbar ) \mathcal{U}_1(t) dt \\[0.2cm]
 & \quad + \int_0^1 \frac{(1-t)^{N-1}}{(N-1)!} \left( \frac{i}{\hbar} \right)^{N} \mathcal{U}_1(t)^* \Ad^{N}_{\widehat{F}_{1,\hbar}}(  \widehat{V}_\hbar ) \mathcal{U}_1(t) dt .
\end{align*}
Let us now set $F_1$ to be the solution of the cohomological equation
\begin{equation}
\label{cohomological1}
\frac{i}{\hbar} [\widehat{F}_{1,\hbar}, \widehat{H}_\hbar ] = \langle \widehat{V}_\hbar \rangle - \widehat{V}_\hbar,
\end{equation}
where the quantum average $\langle \widehat{V}_\hbar \rangle$ (recall Proposition \ref{p:quantum_average}) is given by
$$
\langle \widehat{V}_\hbar \rangle = \lim_{T \to \infty} \frac{1}{T} \int_0^T e^{-i\frac{t}{\hbar}\widehat{H}_\hbar} \, \widehat{V}_\hbar \, e^{i\frac{t}{\hbar}\widehat{H}_\hbar} dt = \Op_\hbar(\langle V \rangle).
$$
Note that the commutator in the left is equal to $\Op_\hbar(\{F_1,H \})$, since $H$ is a polynomial of degree two, hence equation \eqref{cohomological1} at symbol level is just
\begin{equation}
\label{cohomological-classic1}
\{ H, F_1 \} = V - \langle V \rangle.
\end{equation}
Then a solution $F_1 \in S^0(\R^{2d})$ with $\langle F_1 \rangle =0$ exists provided that $\omega$ is partially Diophantine (see Section \ref{cohomological-equations}):
\begin{equation}
\label{solution-cohomological-1}
F_1 =  \frac{1}{(2\pi)^{d_\omega}}\sum_{k \in \mathbb{Z}^{d_\omega} \setminus \{0\}} \frac{V_k}{ik \cdot \omega}, \quad V_k := \int_{\mathbb{T}^{d_\omega}} V \circ \Phi_\tau^\cH e^{-ik \cdot \tau} d\tau.
\end{equation}
In the periodic case, $\omega = (1, \ldots , 1)$, $F_1$ has the simpler expression
\begin{equation}
\label{solution-cohomological-2}
F_1 = -\frac{1}{2\pi} \int_0^{2\pi} \int_0^t (V - \langle V \rangle) \circ \phi^H_s \, ds \, dt,
\end{equation}
given by  \eqref{e:solution_periodic_case}. Thus
\begin{equation}
\label{conjugation-1}
\widehat{P}_\hbar^1 = \mathcal{U}_1^* \, \widehat{P}_\hbar^0 \, \mathcal{U}_1 = \widehat{H}_\hbar + \varepsilon_\hbar \langle \widehat{V}_\hbar \rangle + \sum_{j=2}^N \varepsilon_\hbar^j \widehat{R}_{j,\hbar}^1 +  \varepsilon_\hbar^{N+1} \widehat{T}_\hbar.
\end{equation}
The remainder terms $\widehat{R}_{j,\hbar}^1$ in (\ref{conjugation-1}) can be computed explicitely:
\begin{equation}
\label{e:first_remainder}
\widehat{R}_{j,\hbar}^1 = \frac{1}{j!} \left( \frac{i}{\hbar} \right)^{j-1} \Ad_{\widehat{F}_{1,\hbar}}^{j-1} \left( \langle \widehat{V}_\hbar \rangle + (j-1) \widehat{V}_\hbar \right), \quad j=2, \ldots , N.
\end{equation}
Using equation (\ref{cohomological1}), one can also deduce the following formula for $\widehat{T}_\hbar$:
\begin{align*}
\widehat{T}_\hbar  = \int_0^1 \frac{(1-t)^{N-1}}{N!} \left( \frac{i}{\hbar} \right)^{N} \mathcal{U}_1(t)^* \Ad^{N}_{\widehat{F}_{1,\hbar}}\Big( (1-t)\langle \widehat{V}_\hbar \rangle + (N-1 - t) \widehat{V}_\hbar  \Big) \mathcal{U}_1(t) dt.
\end{align*}
Moreover, by the symbolic calculus for Weyl pseudodifferential operators, the Calderón-Vaillancourt theorem and the fact that $\mathcal{U}_1(t)$ is unitary, we have $\widehat{T}_\hbar = O_{\mathcal{L}(L^2)}(1)$. Iterating this process, we obtain $\widehat{F}_\hbar^j$ and $R_{k,\hbar}^j$ for $2 \leq j \leq N$ and $j \leq k \leq N$ so that, defining $\mathcal{U}_j = \exp[ - i\frac{t \varepsilon_\hbar^j \widehat{F}_{j,\hbar}}{\hbar}]$,  $U_{N,\hbar} = \mathcal{U}_1 \cdots \mathcal{U}_N$, and $R_{j,\hbar} := R_{j,\hbar}^j$, we obtain \eqref{quantum_normal_form}. In particular, 
$$
\widehat{R}_{2,\hbar} =  \frac{i}{2\hbar} [\widehat{F}_{1,\hbar}, \langle \widehat{V}_\hbar \rangle + \widehat{V}_\hbar],
$$
and, in the case $\omega = (1, \ldots, 1)$, we obtain that $\langle \widehat{R}_{2,\hbar} \rangle = \Op_\hbar(\langle L\rangle) + O(\hbar)$ where $\langle L \rangle$ is given by \eqref{e:simpler_form}.

On the other hand, (\ref{FIO-estimate}) holds since
$$
\widehat{U}_{N,\hbar} \Op_\hbar(a) \widehat{U}_{N,\hbar}^* = \Op_\hbar(a) - \frac{i\varepsilon_\hbar}{\hbar}[\widehat{F}_{1,\hbar}, \Op_\hbar(a)] +O_{\mathcal{L}(L^2)}(\varepsilon_\hbar^2),
$$
and $[\widehat{F}_{1,h}, \Op_\hbar(a)] = O_{\mathcal{L}(L^2)}(\hbar)$ for all $a \in \mathcal{C}_c^\infty(\R^{2d})$.

Finally, to show \eqref{e:normal_form_diphantine_case} when the vector of frequencies $\omega$ is Diophantine, we just notice that, by point (i) of Proposition \ref{p:quantum_average}, for every $0 \leq j \leq k(N-1)$, there exists $f_{j,\hbar} \in S^0(\R^d)$ such that
$$
\langle r_j \rangle = f_{j,\hbar}(\Op_\hbar(H_1), \ldots, \Op_\hbar(H_d)),
$$
and moreover $f_{j,\hbar} = f_j + O(\hbar)$, where $f_j$ is given by point (i) of Proposition \ref{p:av_prop}. Regrouping the terms of same order $\hbar^j$, we obtain \eqref{e:normal_form_diphantine_case}.
\end{proof}

\input{bib_Quasimodes.bbl}
\end{document}

%% file: def.tex
\usepackage{dsfont}
\usepackage{centernot}

\newcommand{\nwc}{\newcommand}
\nwc{\nwt}{\newtheorem}
\nwt{theo}{Theorem}
\nwt{coro}[theo]{Corollary}
\nwt{ex}[theo]{Example}
\nwt{prop}[theo]{Proposition}
\nwt{defin}[theo]{Definition}
\nwt{rem}[theo]{Remark}

\nwc{\mf}{\mathbf} 
\nwc{\blds}{\boldsymbol} 
\nwc{\ml}{\mathcal} 

\nwc{\lam}{\lambda}
\nwc{\del}{\delta}
\nwc{\Del}{\Delta}
\nwc{\Lam}{\Lambda}
\nwc{\elll}{\ell}

\nwc{\IA}{\mathbb{A}} 
\nwc{\IB}{\mathbb{B}} 
\nwc{\IC}{\mathbb{C}} 
\nwc{\ID}{\mathbb{D}} 
\nwc{\IE}{\mathbb{E}} 
\nwc{\IF}{\mathbb{F}} 
\nwc{\IG}{\mathbb{G}} 
\nwc{\IH}{\mathbb{H}} 
\nwc{\IN}{\mathbb{N}} 
\nwc{\IP}{\mathbb{P}} 
\nwc{\IQ}{\mathbb{Q}} 
\nwc{\IR}{\mathbb{R}} 
\nwc{\IS}{\mathbb{S}} 
\nwc{\IT}{\mathbb{T}} 
\nwc{\IZ}{\mathbb{Z}} 

\def\bbleft{{\mathchoice {[\mskip-3mu {[}} {[\mskip-3mu {[}}{[\mskip-4mu {[}}{[\mskip-5mu {[}}}}
\def\bbright{{\mathchoice {]\mskip-3mu {]}} {]\mskip-3mu {]}}{]\mskip-4mu {]}}{]\mskip-5mu {]}}}}
\nwc{\setK}{\bbleft 1,K \bbright}
\nwc{\setN}{\bbleft 1,\cN \bbright}
 
\nwc{\va}{{\bf a}}
\nwc{\vb}{{\bf b}}
\nwc{\vc}{{\bf c}}
\nwc{\vd}{{\bf d}}
\nwc{\ve}{{\bf e}}
\nwc{\vf}{{\bf f}}
\nwc{\vg}{{\bf g}}
\nwc{\vh}{{\bf h}}
\nwc{\vi}{{\bf i}}
\nwc{\vI}{{\bf I}}
\nwc{\vj}{{\bf j}}
\nwc{\vk}{{\bf k}}
\nwc{\vl}{{\bf l}}
\nwc{\vm}{{\bf m}}
\nwc{\vM}{{\bf M}}
\nwc{\vn}{{\bf n}}
\nwc{\vo}{{\it o}}
\nwc{\vp}{{\bf p}}
\nwc{\vq}{{\bf q}}
\nwc{\vr}{{\bf r}}
\nwc{\vs}{{\bf s}}
\nwc{\vt}{{\bf t}}
\nwc{\vu}{{\bf u}}
\nwc{\vv}{{\bf v}}
\nwc{\vw}{{\bf w}}
\nwc{\vx}{{\bf x}}
\nwc{\vy}{{\bf y}}
\nwc{\vz}{{\bf z}}
\nwc{\bal}{\blds{\alpha}}
\nwc{\bep}{\blds{\epsilon}}
\nwc{\barbep}{\overline{\blds{\epsilon}}}
\nwc{\bnu}{\blds{\nu}}
\nwc{\bmu}{\blds{\mu}}
\nwc{\bet}{\blds{\eta}}

\nwc{\bk}{\blds{k}}
\nwc{\bm}{\blds{m}}
\nwc{\bM}{\blds{M}}
\nwc{\bp}{\blds{p}}
\nwc{\bq}{\blds{q}}
\nwc{\bn}{\blds{n}}
\nwc{\bv}{\blds{v}}
\nwc{\bw}{\blds{w}}
\nwc{\bx}{\blds{x}}
\nwc{\bxi}{\blds{\xi}}
\nwc{\by}{\blds{y}}
\nwc{\bz}{\blds{z}}

\nwc{\cA}{\ml{A}}
\nwc{\cB}{\ml{B}}
\nwc{\cC}{\ml{C}}
\nwc{\cD}{\ml{D}}
\nwc{\cE}{\ml{E}}
\nwc{\cF}{\ml{F}}
\nwc{\cG}{\ml{G}}
\nwc{\cH}{\ml{H}}
\nwc{\cI}{\ml{I}}
\nwc{\cJ}{\ml{J}}
\nwc{\cK}{\ml{K}}
\nwc{\cL}{\ml{L}}
\nwc{\cM}{\ml{M}}
\nwc{\cN}{\ml{N}}
\nwc{\cO}{\ml{O}}
\nwc{\cP}{\ml{P}}
\nwc{\cQ}{\ml{Q}}
\nwc{\cR}{\ml{R}}
\nwc{\cS}{\ml{S}}
\nwc{\cT}{\ml{T}}
\nwc{\cU}{\ml{U}}
\nwc{\cV}{\ml{V}}
\nwc{\cW}{\ml{W}}
\nwc{\cX}{\ml{X}}
\nwc{\cY}{\ml{Y}}
\nwc{\cZ}{\ml{Z}}

\nwc{\fA}{\mathfrak{a}}
\nwc{\fB}{\mathfrak{b}}
\nwc{\fC}{\mathfrak{c}}
\nwc{\fD}{\mathfrak{d}}
\nwc{\fE}{\mathfrak{e}}
\nwc{\fF}{\mathfrak{f}}
\nwc{\fG}{\mathfrak{g}}
\nwc{\fH}{\mathfrak{h}}
\nwc{\fI}{\mathfrak{i}}
\nwc{\fJ}{\mathfrak{j}}
\nwc{\fK}{\mathfrak{k}}
\nwc{\fL}{\mathfrak{l}}
\nwc{\fM}{\mathfrak{m}}
\nwc{\fN}{\mathfrak{n}}
\nwc{\fO}{\mathfrak{o}}
\nwc{\fP}{\mathfrak{p}}
\nwc{\fQ}{\mathfrak{q}}
\nwc{\fR}{\mathfrak{r}}
\nwc{\fS}{\mathfrak{s}}
\nwc{\fT}{\mathfrak{t}}
\nwc{\fU}{\mathfrak{u}}
\nwc{\fV}{\mathfrak{v}}
\nwc{\fW}{\mathfrak{w}}
\nwc{\fX}{\mathfrak{x}}
\nwc{\fY}{\mathfrak{y}}
\nwc{\fZ}{\mathfrak{z}}

\nwc{\tA}{\widetilde{A}}
\nwc{\tB}{\widetilde{B}}
\nwc{\tE}{E^{\vareps}}
\nwc{\tk}{\tilde k}
\nwc{\tN}{\tilde N}
\nwc{\tP}{\widetilde{P}}
\nwc{\tQ}{\widetilde{Q}}
\nwc{\tR}{\widetilde{R}}
\nwc{\tV}{\widetilde{V}}
\nwc{\tW}{\widetilde{W}}
\nwc{\ty}{\tilde y}
\nwc{\teta}{\tilde \eta}
\nwc{\tdelta}{\tilde \delta}
\nwc{\tlambda}{\tilde \lambda}
\nwc{\ttheta}{\tilde \theta}
\nwc{\tvartheta}{\tilde \vartheta}
\nwc{\tPhi}{\widetilde \Phi}
\nwc{\tpsi}{\tilde \psi}
\nwc{\tmu}{\tilde \mu}

\nwc{\To}{\longrightarrow} 

\nwc{\ad}{\rm ad}
\nwc{\eps}{\varepsilon}
\nwc{\ep}{\epsilon}
\nwc{\vareps}{\varepsilon}

\def\ep{\epsilon}
\def\tr{{\rm tr}}

\def\sq2{\sqrt{2}}

\def\t2{{\mathbb T}^2}
\def\s2{{\mathbb S}^2}

\def\N{\mathbb{N}}
\def\T{\mathbb{T}}
\def\R{\mathbb{R}}

\def\Z{\mathbb{Z}}
\def\C{\mathbb{C}}

\nwc{\lap}{\bigtriangleup}
\nwc{\rest}{\restriction}
\nwc{\Diff}{\operatorname{Diff}}
\nwc{\diam}{\operatorname{diam}}
\nwc{\Res}{\operatorname{Res}}
\nwc{\Spec}{\operatorname{Sp}}
\nwc{\Vol}{\operatorname{Vol}}
\nwc{\Op}{\operatorname{Op}}
\nwc{\supp}{\operatorname{supp}}
\nwc{\Span}{\operatorname{span}}
\nwc{\weaksc}{\overset{\ast}{\rightharpoonup}}
\nwc{\charf}{\mathds{1}}
\nwc{\notimplies}{\centernot\implies}

\nwc{\dia}{\varepsilon}
\nwc{\cut}{f}
\nwc{\qm}{u_\hbar}

\def\hto0{\xrightarrow{\hbar\to 0}}

\def\rto0{\xrightarrow{r\to 0}}

\nwc{\la}{\langle}
\nwc{\ra}{\rangle}
\nwc{\lp}{\left(}
\nwc{\rp}{\right)}

\nwc{\bequ}{\begin{equation}}
\nwc{\be}{\begin{equation}}
\nwc{\ben}{\begin{equation*}}
\nwc{\bea}{\begin{eqnarray}}
\nwc{\bean}{\begin{eqnarray*}}
\nwc{\bit}{\begin{itemize}}
\nwc{\bver}{\begin{verbatim}}

\nwc{\eequ}{\end{equation}}
\nwc{\ee}{\end{equation}}
\nwc{\een}{\end{equation*}}
\nwc{\eea}{\end{eqnarray}}
\nwc{\eean}{\end{eqnarray*}}
\nwc{\eit}{\end{itemize}}
\nwc{\ever}{\end{verbatim}}